\documentclass{amsart}
\usepackage{hyperref}
\usepackage{newlfont}
\usepackage{graphicx}
\usepackage{indentfirst}
\usepackage{fancyhdr}
\usepackage{amsmath}
\usepackage{amsthm}
\usepackage{latexsym}
\usepackage[psamsfonts]{amssymb}
\usepackage{mathrsfs}
\usepackage{frcursive}
\usepackage
{lineno}
\usepackage{lipsum}
\usepackage{calligra}
\usepackage[T1]{fontenc}
\usepackage{wesa}
\usepackage{extarrows}
\usepackage{tcolorbox}
\usepackage[utf8]{inputenc}
\usepackage[english]{babel}
\usepackage{mathtools}
\usepackage{adjustbox,expl3,etoolbox}
\usepackage{marginnote}

 \usepackage{relsize}
 \usepackage{textgreek}

\usepackage{accents}
\newcommand{\dbtilde}[1]{\accentset{\approx}{#1}}
\newcommand{\bigzero}{\mbox{\normalfont\Large\bfseries 0}}

\usepackage{adjustbox,expl3,etoolbox}

\letcs\replicate{prg_replicate:nn}

\newcommand*\longsum[1][1]{%
	\mathop{\textnormal{%
			\clipbox{0pt 0pt {.5\width} 0pt}{$\displaystyle\sum$}%
			\replicate{#1}{\clipbox{{.5\width} 0pt {.4\width} 0pt}{$\displaystyle\sum$}}%
			\clipbox{{.6\width} 0pt 0pt 0pt}{$\displaystyle\sum$}}}%
}
\usepackage{tikz}

\makeatletter
\@namedef{subjclassname@2020}{\textup{2020} Mathematics Subject Classification}
\makeatother

\allowdisplaybreaks
\newtheorem{theorem}{Theorem}
\newtheorem{proposition}{Proposition}
\newtheorem{lemma}{Lemma}

\newtheorem{definition}{Definition}
\newtheorem{example}{Example}
\newcounter{obsctr}

\newtheorem{remark}{Remark}

\renewcommand{\theequation}{\thesection.\arabic{equation}}

\DeclareMathAlphabet{\mathcalligra}{T1}{calligra}{m}{n}
\DeclareFontShape{T1}{calligra}{m}{n}{<->s*[2.2]callig15}{}

\begin{document}
\title[On a class of globally analytic Hypoelliptic operators]{On a class of globally analytic Hypoelliptic operators
	with non-negative characteristic form}
\author{N. Braun Rodrigues}
\email{braun.rodrigues@gmail.com}
\address{Instituto de Matem\'atica, Estat\'istica e Computação Científica da Unicamp, Rua Sérgio Buarque de Holanda, 651, 3083-859, Campinas, SP, Brazil}
\author{G. Chinni}
\email{gregorio.chinni3@unibo.it/gregorio.chinni@gmail.com}
\address{Dipartimento di Matematica, Universit\`a
	di Bologna, Piazza di Porta San Donato 5, Bologna,
	Italy}
%
%
\date{\today}
%
%
\begin{abstract}
	  The global analytic hypoellipticity is proved for a class
	  of second order partial differential equations with non-negative
	  characteristic form globally defined on the torus. The class considered in this work  generalizes at some degree the class of sum of squares considered by Bove-Chinni and also by Cordaro-Himonas.
\end{abstract}
%
\maketitle
\tableofcontents
\section{Introduction}
\renewcommand{\theequation}{\thesection.\arabic{equation}}
\setcounter{equation}{0} \setcounter{theorem}{0}
\setcounter{proposition}{0} \setcounter{lemma}{0}
\setcounter{corollary}{0} \setcounter{definition}{0}
\setcounter{remark}{1}
%
We deal with the global analytic hypoellipticity
for second order operators with non-negative
characteristic form on the $N$ dimensional torus, generalizing at some degree
the work \cite{BC-2022} concerning sums of squares of vector fields. 
The present paper generalize \cite{BC-2022} in the sense that our class of operators
contains  operators that can not be written as sums of squares,
but our hypothesis are less general as the ones in \cite{BC-2022} in the case of  sums of squares (see Remark \ref{Rk:Appendix} in  the Appendix).
We consider operator of the form
\begin{linenomath}
	\begin{align}\label{HOR_Op}
	P(x,D) = \sum_{\ell,j=1}^{N} a_{\ell,j}(x) D_{\ell} D_{j} + i  \sum_{\ell=1}^{N}  b_{\ell}(x) D_{\ell} +c(x) 
	\end{align}
\end{linenomath}
on $\mathbb{T}^{N}$, where $a_{\ell,j}(x) $, $b_{\ell}(x)$, $\ell,j \in \lbrace 1,\dots,\, N\rbrace$,
are real valued real analytic functions,
the matrix $\mathbf{A}(x)= \left( a_{\ell,j}(x)\right)$ is real symmetric,
$a_{\ell,j}(x) = a_{j,\ell}(x)$, and $A(x) \geq 0$ on $\mathbb{T}^{N}$
(i.e. $\langle \mathbf{A}(x)v, v\rangle \geq 0$ for any $v \in \mathbb{C}^{N}$ and $x \in \mathbb{T}^{N}$)
and  $c(x)$ is an analytic function on $\mathbb{T}^{N}$.\\
We recall that in the local setting (i.e. on open subset of $\mathbb{R}^{N}$) and in more general situation, i.e. when the coefficients
are smooth functions, this class of operators, with non-negative
characteristic form, was intensively studied by O.A.~Ole\u \i nik and  E.V.~Radkevi\v c in \cite{OR1971}. 
We will assume that $P$, \eqref{HOR_Op}, satisfies the \textit{H\"ormander-Ole\u{\i}nik-Radkevi\v c} condition,
see Definition \ref{HOR-c}. As known this class of operators contains the H\"ormander's  class of first kind,
 \cite{H67}, i.e. when $P$ can be written as a sum of squares:
$P=\sum_{j=1}^{r} X_{j}^{2}+ iX_{0}$, where $X_{j}$, $j=0,1,\dots,r$, are vector fields with real analytic coefficients,
and the Lie algebra is generated by the family $X_{j}$, $j=1,\dots,r$.  
We recall that in general an operator of the form \eqref{HOR_Op} can not be written as sum of square,
see Example \ref{Ex1} in the Appendix.\\
We are concerned with the regularity of the solution of the problem $Pu=f$, $f$ analytic function on $\mathbb{T}^{N}$.
In general is not expected that $u$ is analytic, i.e. that $P$ is globally analytic hypoelliptic on $\mathbb{T}^{N}$ (see Definition \ref{D_GGH}.) 
Indeed it is enough to consider the global version of the M\'etivier operator
\footnote{The local vesion of this operator was studied in detail by M\'etivier in \cite{M81}. }
\begin{linenomath}
	\begin{equation}
	\label{eq:Metop}
	M = D_{x}^{2} + (\sin x)^{2} D_{y}^{2} + ( \sin y \  D_{y})^{2} .
	\end{equation}
\end{linenomath}
It was showed in \cite{trevespienza}, see also \cite{BC-2022}, that $M$
is globally Gevrey 2 hypoelliptic on the two dimensional torus, $ \mathbb{T}^{2} $, and not better. 
We remark that $M$ has a characteristic variety which is not symplectic, the Hamilton
leave lies along the $\eta$ fiber of the cotangent bundle (here $\eta$ is the dual variable of $y$.)
We recall that the fact that the characteristic variety is symplectic
is not enough to guarantee the analytic hypoellipticity.
It is sufficient in fact to consider the following generalization of the M\'etivier operator
\begin{linenomath}
	\begin{equation}
	\label{eq:Metop2}
	M_{p,a} =  D_{x}^{2} + (\sin x)^{2(2p-1)} D_{y}^{2} + ((\sin x)^{p-1}
	(\sin y)^{a} \  D_{y})^{2} 
	\end{equation}
\end{linenomath}
where $ p > 1  $, $ a \geq 1 $. In \cite{chinni23-1}, see also\cite{chinni23-2}, was showed that $M_{p,a}$
is globally Gevrey $\frac{2a}{2a-1}$ hypoelliptic on the torus $ \mathbb{T}^{2} $ and not better. 
As stated in \cite{BC-2022} and in some sense conjectured in \cite{trevespienza} and \cite{BoTr-04},
it seems reasonable to surmise that the only globally non analytic
hypoelliptic operators are operators where the ``M\'etivier
situation'' occurs, meaning that there is a Hamilton leaf lying along
i.e. the presence of a non compact Hamilton leaf in the characteristic variety of the operator.
We point out that, at the present, already in the case of sum of squares operators,
there isn't a solid theory that allow to detect the presence of a Hamilton leaf,
at least in dimension greater or equal to 3, see \cite{ABM-16}.\\
As done in \cite{BC-2022}, in order to avoid the presence M\'etivier type of phenomenon
we do some assumptions on the coefficients of $P$,
assumption $\mathbf{(A3)}$, in particular \eqref{A3_1},
which allow us to rule out the existence of a ``Hamilton leaf'' on the characteristic variety
lying along the fiber of the cotangent bundle, i.e. the case of the (global) M\'etivier operator.
We point out that in the case of sums of squares these assumptions are stronger than that
done in \cite{BC-2022}, we believe that this partly depends from the possibility of writing
the operator \eqref{HOR_Op} in the form of sums of squares.
We conclude by observing that the class of operators studied in the present paper
intersects those studied by Cordaro and Himonas in \cite{ch94} (see also \cite{Christ95}, \cite{Tart_96},
\cite{HP2}, \cite{CC-16}, and on compact Lie groups in \cite{bj_2024})
and by Bove and the second author in \cite{BC-2022}.
%
\section{Notations, Definitions, Preliminary facts and Main Results}
\renewcommand{\theequation}{\thesection.\arabic{equation}}
\setcounter{equation}{0} \setcounter{theorem}{0}
\setcounter{proposition}{0} \setcounter{lemma}{0}
\setcounter{corollary}{0} \setcounter{definition}{0}
\setcounter{remark}{1}
Let $\mathbb{T}^{N} = \mathbb{R}^{N}/(2\pi \mathbb{Z}^{N})$ be the $N$-dimensional torus,
where the coordinates are written as $x=(x_{1}, \, \dots, x_{N})$. We denote by $G^{s}(\mathbb{T}^{N})$,
$s\geq 1$, the space of Gevrey functions of order $s$ on $\mathbb{T}^{N}$. In particular
$G^{1}(\mathbb{T}^{N})= C^{\omega}(\mathbb{T}^{N})$ is the space of real-analytic functions on $\mathbb{T}^{N}$.\\
We recall that $u$ belongs to $G^{s}(\mathbb{T}^{N})$ if and only if $u \in C^{\infty}(\mathbb{T}^{N})$
and there is a positive constant $C$ such that
\begin{linenomath}
	\begin{equation}\label{Gs_f}
	\| D^{\alpha} u\|_{0} \leq C^{|\alpha|+1} \left(\alpha!\right)^{s}, \qquad \forall\, \alpha \in \mathbb{Z}^{N}_{+}. 
	\end{equation}
\end{linenomath}
Here, as usual, $D^{\alpha}= D_{1}^{\alpha_{1}} D_{2}^{\alpha_{2}}\cdots D_{N}^{\alpha_{N}}$,
where $D_{j}= \frac{1}{\sqrt{-1}}\frac{\partial}{\partial x_{j}}$.\\
In terms of Fourier coefficients, for a function $u \in C^{\infty}(\mathbb{T}^{N})$ to belong
to $G^{s}(\mathbb{T}^{N})$ it is necessary and sufficient that there is a positive constant $\delta$
such that 
\begin{linenomath}
	\begin{equation}\label{Gs_f2}
	\| e^{\delta |\xi|^{1/s} }\, \widehat{u}(\xi)\|_{\ell_{\infty}(\mathbb{Z}^{N})} < \infty. 
	\end{equation}
\end{linenomath}
We write the Fourier coefficients of $u$ as
\begin{linenomath}
	\begin{equation}\label{fourier_c}
	u(x) = \longsum[5]_{\xi \in \mathbb{Z}^{N}} e^{i\xi x}\, \widehat{u}(\xi), \quad
	\widehat{u}(\xi) = \frac{1}{(2\pi)^{N}}\int_{\mathbb{T}^{N}} e^{-i\xi x} u(x) \, dx. 
	\end{equation}
\end{linenomath}
Sobolev spaces in $\mathbb{T}^{N}$: if $s \in \mathbb{R}$, we denote by $H^{s}(\mathbb{T}^{N})$
the space of all $u \in \mathscr{D}'(\mathbb{T}^{N})$ such that 
\begin{linenomath}
	\begin{equation}\label{Sobolev_1}
	\bigg\{\left(1+ |\xi|^{2}\right)^{\frac{s}{2}}  \widehat{u}(\xi)\bigg\}_{\xi \in \mathbb{Z}^{N}} \in \ell_{2} (\mathbb{Z}^{N}). 
	\end{equation}
\end{linenomath}
The space $H^{s}(\mathbb{T}^{N})$ are Hilbert space equipped with the norm
\begin{linenomath}
	\begin{equation}\label{Sobolev_2}
	\|  u\|_{s}  = \bigg[ \sum_{\xi \in \mathbb{Z}^{N}} \left(1+|\xi|^{2}\right)^{s} |\widehat{u}(\xi)|^{2}\bigg]^{1/2};
	\end{equation}
\end{linenomath}
$\|\cdot\|_{0}$ is the $L^{2}$-norm. Here the space $L^{2}(\mathbb{T}^{N})$ will be endowed
with the inner product
\begin{linenomath}
	\begin{equation}\label{Ptod_L2}
	\langle u,\, v \rangle = \frac{1}{(2\pi)^{N}}\int_{\mathbb{T}^{N}} u(x)\, \overline{v(x)} \, dx. 
	\end{equation}
\end{linenomath}
We recall the definition of global Gevrey/analytic hypoellipticity
\begin{definition}\label{D_GGH}
	Let $P$ be a differential operator on the $N$-dimensional torus. $P$ is said to be
	globally $G^{s}$-hypoelliptic, $s\geq 1$, in $\mathbb{T}^{N}$ if the conditions
	$u \in \mathscr{D}'(\mathbb{T}^{N})$ and $Pu \in G^{s}(\mathbb{T}^{N})$ imply that $ u\in G^{s}(\mathbb{T}^{N})$.
\end{definition}
%
\subsection{Basic Estimate}
The purpose of this paragraph is to obtain the global version
of the estimate obtained, in the local setting, in \cite{D_2020}.
It will be the main tool to obtain our results.  
We consider the operator \eqref{HOR_Op}. 
We associate to $P$ the family of differential operators
$P^{k}$ and $P_{k}$, $k=1,\,\dots,\, N$, of order $1$ and $2$ respectively,
of the following form
\begin{linenomath}
	\begin{align}\label{Pk_ud}
	\begin{cases}
	P^{k}(x,D) = 2 \displaystyle\sum_{\ell =1}^{N} a_{\ell,k}(x) D_{\ell}, 
	&\\
	\noalign{\vskip6pt}
	P_{k}(x,D) = \displaystyle\longsum[4]_{\ell,j =1}^{N}a_{\ell,j}^{(k)}(x)D_{\ell}D_{j},
	&
	\end{cases}
	\end{align}
\end{linenomath}
where $a_{\ell,j}^{(k)}(x) = D_{k}a_{\ell,j}(x)$.
%
\begin{lemma}\label{L1}
	Let $P^{k}(x,D)$ be as in \eqref{Pk_ud}. Then there is a  positive constant $C$
	such that
	\begin{linenomath}
		\begin{align}\label{BEst_1}
		\sum_{k=1}^{N}\| P^{k} u\|_{0}^{2} \leq C\left( |\langle Pu,u \rangle | +\|u\|_{0}^{2}\right), \qquad\forall u\in C^{\infty}(\mathbb{T}^{N}).
		\end{align}
	\end{linenomath}
\end{lemma}
\begin{proof}
We observe that
\begin{linenomath}
	\begin{multline*}
      \langle Pu, u \rangle 
      = \longsum[5]_{j,\ell =1}^{N} \langle a_{j,\ell} D_{j} D_{\ell} u, u \rangle
      + i \sum_{\ell}^{N} \langle b_{\ell} D_{\ell} u, u \rangle + \langle c u, u\rangle
      \\
      =
      \longsum[5]_{j,\ell =1}^{N} \langle a_{j,\ell} D_{j}  u,  D_{\ell} u \rangle
      - \longsum[5]_{j,\ell =1}^{N} \langle a_{j,\ell}^{(\ell)} D_{j}  u,  u \rangle
      + i \sum_{\ell}^{N} \langle b_{\ell} D_{\ell} u, u \rangle + \langle c u, u\rangle,
	\end{multline*}
\end{linenomath}
and 
\begin{linenomath}
	\begin{multline*}
	\langle u, P u \rangle 
	=
	\longsum[5]_{j,\ell =1}^{N} \langle a_{j,\ell} D_{j}  u,  D_{\ell} u \rangle
	+
	\longsum[5]_{j,\ell =1}^{N} \langle a_{j,\ell}^{(\ell)} D_{j} u, u \rangle
	\\
	+
	\longsum[5]_{j,\ell =1}^{N} \langle a_{j,\ell}^{(j+\ell)}   u,  u \rangle
	- i \sum_{\ell}^{N} \langle b_{\ell} D_{\ell} u, u \rangle 
	- i \sum_{\ell}^{N} \langle b_{\ell}^{(\ell)}  u, u \rangle 
	+ \langle \overline{c} u, u\rangle,
	\end{multline*}
\end{linenomath}
where we used that $a_{j,\ell} = a_{\ell,j}$.
We conclude that
\begin{linenomath}
	\begin{multline}\label{eq:ReP}
	2\text{Re}  \langle Pu, u \rangle 
	= 
	2\longsum[5]_{j,\ell =1}^{N} \langle a_{j,\ell} D_{j}  u,  D_{\ell} u \rangle
	+
	\longsum[5]_{j,\ell =1}^{N} \langle a_{j,\ell}^{(j+\ell)}   u,  u \rangle
	- i \sum_{\ell}^{N}  \langle b_{\ell}^{(\ell)}  u, u \rangle 
	\\
	+ \langle (c+\overline{c}) u, u\rangle.
	\end{multline}
\end{linenomath}
Remarking that $2 \displaystyle\sum_{\ell =1}^{N}\sum_{j=1}^{N} a_{j,\ell} D_{\ell} = \displaystyle\sum_{\ell =1}^{N} P^{\ell}$,
the above identity can be rewritten as
\begin{linenomath}
	\begin{align*}
	\sum_{\ell =1}^{N} \langle P^{\ell} u, D_{\ell} u\rangle
	=
	2\text{Re}  \langle Pu, u \rangle
	-
	\longsum[5]_{j,\ell =1}^{N} \langle a_{j,\ell}^{(j+\ell)}   u,  u \rangle
	+ i \sum_{\ell}^{N} \langle b_{\ell}^{(\ell)}  u, u \rangle 
	- \langle (c+\overline{c}) u, u\rangle.
	\end{align*}
\end{linenomath}
We obtain
\begin{linenomath}
	\begin{align}\label{Est_0}
	\bigg|\sum_{\ell =1}^{N} \langle P^{\ell} u, D_{\ell} u\rangle \bigg|
	\leq 
	2
	|\text{Re}  \langle Pu, u \rangle| + \tilde{C} \|u \|^{2}_{0}
	\leq 
	C_{1}
	\left( | \langle Pu, u \rangle| + \|u \|_{0}^{2} \right).
	\end{align}
\end{linenomath}
On the other hand we have
\begin{linenomath}
	\begin{multline*}
	\sum_{k =1}^{N} \| P^{k} u\|^{2}
	=
	4 \sum_{k =1}^{N} \bigg\langle \sum_{\ell =1}^{N} a_{\ell,k} D_{\ell} u, \sum_{j=1}^{N} a_{j,k} D_{j} u \bigg\rangle
    \\
	=
	4 \sum_{\ell =1}^{N} \sum_{j=1}^{N} \bigg\langle \sum_{k =1}^{N} a_{j,k} a_{\ell,k} D_{\ell} u,  D_{j} u\bigg\rangle
	= 4 \sum_{\ell =1}^{N} \sum_{j=1}^{N} \bigg\langle \left(\sum_{k =1}^{N} a_{j,k} a_{k,\ell}\right) D_{\ell} u,  D_{j} u\bigg\rangle,
	\end{multline*}
\end{linenomath}
where we use that $ a_{\ell,k} = a_{k,\ell}$.
$\sum_{k =1}^{N} a_{j,k} a_{k,\ell}$ are the entries of the matrix $\mathbf{A}^{2}(x)$. 
Since $\mathbf{A}(x)$ is positive semidefinite, 
there is a suitable positive constant, $C_{0}$, such that
$\langle \mathbf{A}^{2} v, v \rangle \leq C_{0} \langle \mathbf{A} v, v \rangle $.
We obtain 
\begin{linenomath}
	\begin{align*}
	\sum_{k =1}^{N} \| P^{k} u\|^{2}
	\leq 4 C_{0} \longsum[5]_{j,\ell =1}^{N}  \langle a_{j,\ell} D_{j} u,  D_{\ell} u\rangle
	=
	4 C_{0} \sum_{\ell =1}^{N} \langle P^{\ell} u, D_{\ell} u\rangle.
	\end{align*}
\end{linenomath}
Summing up, by the above estimate and \eqref{Est_0}, we conclude that there is a positive constant, $C$,
such that
\begin{linenomath}
	\begin{align*}
	\sum_{k =1}^{N} \| P^{k} u\|^{2}
	\leq
	C\left( | \langle Pu, u \rangle| + \|u \|_{0}^{2} \right),
	\end{align*}
\end{linenomath}
i.e. \eqref{BEst_1}.
\end{proof}
For our purpose we recall the following result due to O.A.~Ole\u \i nik and  E.V.~Radkevi\v c:
%
\begin{lemma}[\cite{OR1971}]\label{OR_L}
	Let $\langle \mathbf{A}(x) \xi, \xi\rangle = \displaystyle\longsum[5]_{k,j=1}^{N} a_{k,j}(x) \xi_{k}\xi_{j} \geq 0$
	for all $x \in \mathbb{R}^{N}$ and all $\xi =(\xi_{1},\, \dots, \xi_{N})$, suppose that $a_{k,j}(x) \in \mathscr{B}^{2}(\mathbb{R}^{N})$,
	the set of $C^{2}$ functions bounded in $\mathbb{R}^{N}$;
	then for any $v \in C^{2}(\mathbb{R}^{N})$ and $\rho \in \{1, \dots, N \}$ the following estimate holds
	\begin{linenomath}
		\begin{align}\label{Matrix_Est}
		\left| 
		\longsum[6]_{k,j=1}^{N} \frac{\partial a_{k,j}}{\partial x_{\rho}} (x) \frac{\partial^{2} v}{\partial x_{k} \partial x_{j}}(x)\right|^{2}
		\leq
		M \longsum[9]_{k,j,s=1}^{N} a_{k,j} (x)   \frac{\partial^{2} v}{\partial x_{k} \partial x_{s}} (x)
		\overline{\frac{\partial^{2} v}{\partial x_{j} \partial x_{s}} (x)}. 
		\end{align}
	\end{linenomath}
\end{lemma}
Taking advantage from the above Lemma we obtain
\begin{lemma}\label{L2}
	Let $P_{k}(x,D)$ be as in \eqref{Pk_ud}. Then there is $C$, positive constant,
	such that
	\begin{linenomath}
		\begin{align}\label{Est_P_k}
		\sum_{k=1}^{N} \| P_{k} u \|_{-1}^{2} 
		\leq
		C \left( \sum_{s=1}^{N} | \langle D_{s} \Lambda^{-1} P  u , D_{s} \Lambda^{-1} u \rangle | + \|u\|_{0}^{2}\right), \quad\forall u\in C^{\infty}(\mathbb{T}^{N}),
		\end{align}
	\end{linenomath}
where $\Lambda^{-1}$ is the pseudodifferential operator of order $-1$ with associated symbol $\lambda(\xi) = (1+ |\xi|^{2})^{-1/2}$. 
\end{lemma}
\begin{proof}
Let $s \in \mathbb{R}$, we set $\Lambda^{s} $ the pseudodifferential operator associated
to the symbol $\lambda(\xi) = \left(1+|\xi|^{2}\right)^{s/2}$, $\xi \in \mathbb{Z}^{n}$,
we have
\begin{linenomath}
	\begin{equation*}
	\sum_{k=1}^{N} \| P_{k} u \|_{-1}^{2} = \sum_{k=1}^{N} \| \Lambda^{-1} P_{k} u \|^{2}
	\leq \sum_{k=1}^{N} \| P_{k} \Lambda^{-1} u \|^{2} + \sum_{k=1}^{N} \| [ \Lambda^{-1}, P_{k}] u \|^{2}.
	\end{equation*}
\end{linenomath}
Since $[ \Lambda^{-1}, P_{k}] = \displaystyle\longsum[4]_{\ell,j =1}^{N}[ \Lambda^{-1}, a_{\ell,j}^{(k)}(x)] D_{\ell}D_{j}$ is a zero order
pseudodifferential operator we obtain
\begin{linenomath}
	\begin{equation*}
	\sum_{k=1}^{N} \| P_{k} u \|_{-1}^{2} = \sum_{k=1}^{N} \| \Lambda^{-1} P_{k} u \|^{2}
	\leq \sum_{k=1}^{N} \| P_{k} \Lambda^{-1} u \|^{2} + C \|u \|^{2},
	\end{equation*}
\end{linenomath}
where $C$ is a suitable positive constant. We set $\widetilde{u}= \Lambda^{-1} u$. By the Lemma \ref{OR_L}
and \eqref{eq:ReP}, we have
\begin{linenomath}
	\begin{multline*}
	 \sum_{k=1}^{N} \| P_{k} \widetilde{u} \|^{2} 
	 = \sum_{k =1}^{N}\int  \Big| \longsum[6]_{j,\ell =1}^{N}  \frac{\partial a_{j,\ell}}{\partial x_{k}}(x) D_{j} D_{\ell} \widetilde{u} \Big|^{2} dx
	\\
	 \leq 
	 MN \int \longsum[8]_{s,j,\ell=1}^{N} a_{j,\ell}(x) \left( D_{j}D_{s} \widetilde{u}\right) \overline{\left( D_{\ell}D_{s} \widetilde{u}\right)} dx
	 \\
	 =
	 MN 
	 \left( 
	 \text{Re}  \sum_{s=1}^{N}\langle P D_{s} \widetilde{u} , D_{s} \widetilde{u} \rangle 
     -\frac{1}{2} \longsum[7]_{s, j,\ell =1}^{N} \langle a_{j,\ell}^{(j+\ell)}   D_{s} \widetilde{u},  D_{s} \widetilde{u} \rangle
     +\frac{1}{2} \longsum[5]_{s,\ell =1}^{N} \langle b_{\ell}^{(\ell)}  D_{s} \widetilde{u}, D_{s} \widetilde{u} \rangle 
     \right.
     \\
     \left.
     - \sum_{s=1}^{N}\langle (c+\overline{c}) D_{s} \widetilde{u}, D_{s} \widetilde{u}\rangle
     \right)
	 .
	\end{multline*}
\end{linenomath}
Since $D_{s} \Lambda^{-1}$ is a zero order pseudodifferential operator, we conclude that there is a suitable
positive constant $C$ such that 
\begin{linenomath}
	\begin{multline}\label{Est:P_k-1}
	\sum_{k=1}^{N} \| P_{k} u \|_{-1}^{2} 
	\leq
	C \left( \sum_{s=1}^{N}  | \langle P D_{s} \Lambda^{-1} u , D_{s} \Lambda^{-1} u \rangle | + \|u\|_{0}^{2}\right)
	\\
	\leq
	C \left( \sum_{s=1}^{N} | \langle D_{s} \Lambda^{-1} P  u , D_{s} \Lambda^{-1} u \rangle | 
	+ \sum_{s=1}^{N} | \langle [ P, D_{s} \Lambda^{-1}] u , D_{s} \Lambda^{-1} u \rangle
	+ \|u\|_{0}^{2} 
	\right)
	.
	\end{multline}
\end{linenomath}
The symbol associated to the operator $ [ P, D_{s} \Lambda^{-1}]$ is of the form
$ \sigma\left( [P, D_{s} \Lambda^{-1}] \right)  = 
\sum_{k=1}^{N} \left[ D_{\xi_{k}} \sigma\left( D_{s} \Lambda^{-1}\right)\right]  \left[ D_{x_{k}} 
\sigma\left( P \right)\right]$ modulo a zero
order symbol. We point out that $ D_{x_{k}}  \sigma\left( P \right) = P_{k}(x,\xi)$. So
\begin{linenomath}
	\begin{multline*}
  \sum_{s=1}^{N} | \langle    [ P, D_{s} \Lambda^{-1} ]  u , D_{s} \Lambda^{-1} u \rangle |
  \leq
  C_{1}\left( \sum_{k=1}^{N} \|   P_{k}  u\|_{-1} \right) \| u \|_{0} + C_{2}\|u\|_{0}^{2}
  \\
   \leq
   \delta N\sum_{k=1}^{N} \|  P_{k}  u\|_{-1}^{2} + C_{\delta}  \|u\|_{0}^{2},
	\end{multline*}
\end{linenomath}
where $\delta$ is a suitable small positive number such that $C\delta \leq 1$, $C$ is the constant
appearing in (\ref{Est:P_k-1}). Such choice allow us to absorb the first term on the right hand side of (\ref{Est:P_k-1}).\\
Summing up we obtain
\begin{linenomath}
	\begin{equation*}
	\sum_{k=1}^{N} \| P_{k} u \|_{-1}^{2} 
	\leq
	\widetilde{C} \left( \sum_{s=1}^{N} | \langle D_{s} \Lambda^{-1} P  u , D_{s} \Lambda^{-1} u \rangle | + \|u\|_{0}^{2}\right)
	.
	\end{equation*}
\end{linenomath}
i.e. (\ref{Est_P_k}).
\end{proof}
By the Lemma \ref{L1} and \ref{L2}, we obtain
\begin{proposition}\label{P1}
	Let $P(x,D)$ be as in \eqref{HOR_Op}. Let $P^{k}(x,D)$ and $P_{k}(x,D)$ be as in \eqref{Pk_ud}. Then there is a positive constant $C$, 
	such that
	\begin{linenomath}
		\begin{align}\label{BEst_3}
		\sum_{k=1}^{N}\| P^{k} u\|_{0}^{2} + \! \sum_{k=1}^{N}\| P_{k} u\|_{-1}^{2}
		\leq C\left(\sum_{s=0}^{N} |\langle E_{s} Pu, E_{s} u \rangle | +\|u\|_{0}^{2}\right), \,\, \forall u\in C^{\infty}(\mathbb{T}^{N}),
		\end{align}
	\end{linenomath}
where $E_{0}=1$ and $E_{s}= D_{s}\Lambda^{-1}$ for $s=1,2,\dots,N$.
\end{proposition}
\begin{remark}
	The following inequality holds
	\begin{linenomath}
		\begin{align}\label{BEst_3_!}
		\sum_{k=1}^{N}\| P^{k} u\|_{0}^{2} + \sum_{k=1}^{N}\| P_{k} u\|_{-1}^{2}
		\leq C_{1}\left( \| Pu \|^{2}_{0} +\|u\|_{0}^{2}\right), \quad\forall u\in C^{\infty}(\mathbb{T}^{N}).
		\end{align}
	\end{linenomath}
\end{remark}
We give now the  global
\textit{H\"ormander-Ole\u{\i}nik-Radkevi\v c condition}.
For more details on the local
\textit{H\"ormander-Ole\u{\i}nik-Radkevi\v c condition} see \cite{Rad2009} and \cite{D_2020}.\\
To the operator \eqref{HOR_Op} we associate the family of operators $\mathcal{Q}=\left\{ Q_{1}, \dots, Q_{2N} \right\} $
where
\begin{linenomath}
	\begin{equation}\label{Eq:Qj}
	Q_{j}=P_{j}  \text{ for } j=1,\dots, N \text{ and } Q_{j}= \Lambda^{-1}P^{j-N}, \, j=N+ 1,\dots, 2N.
	\end{equation}
\end{linenomath}
Let $I= \left( i_{1}, \dots, i_{k} \right)$ be a multi-index where $i_{\ell} \in \left\{1,\dots, 2N\right\}$, $\ell= 1,\dots,k$.
We set $|I|=k$. Given a multi-index $I$, we associate the operator  
\begin{linenomath}
	\begin{align}
	&Q_{I}= Q_{i_{1}}  \qquad \text{if } I= i_{1},
	\\
	\nonumber
	&
	Q_{I}= [Q_{i_{1}},[Q_{i_{2}},[\cdots[Q_{i_{k-1}}, Q_{i_{k}}]\cdots] \qquad \text{if } I=\left( i_{1}, \dots, i_{k} \right).
	\end{align}
\end{linenomath}
\begin{definition}[\textbf{(H.O.R.)-condition}]\label{HOR-c}
	\hspace{20em}\\
	Let $P(x,D)$ be as in \eqref{HOR_Op}. Let $\left\{ Q_{1}, \dots, Q_{2N} \right\}$ be the family of operators associated with $P$,
	\eqref{Eq:Qj}. We say that $P(x,D)$ satisfies the  \textit{H\"ormander-Ole\u{\i}nik-Radkevi\v c} condition, (H.O.R.)-condition for shortness,
	at the point $(x_{0},\xi_{0}) \in \mathbb{T}^{N}\times \mathbb{Z}^{N}\setminus \{0\}$ if there exists $I = (i_{1}, \dots, i_{r})$ such that
	the symbol associated with the operator $Q_{I}$, $\sigma\left(Q_{I}\right)$, is elliptic at $(x_{0},\xi_{0})$, 
	i.e. $\sigma\left(Q_{I}\right)(x_{0},\xi_{0}) \neq 0$.\\
	Let $x_{0} \in \mathbb{T}^{N}$, we say that $P(x,D)$ satisfied the (H.O.R.)-condition at $x_{0}$ if, for every $\xi \in \mathbb{Z}^{N}\setminus \{0\}$,
	the (H.O.R.)-condition is satisfied at $(x_{0},\xi)$. We say that $P(x,D)$ satisfies the (H.O.R.)-condition globally if the (H.O.R.)-condition
	is satisfied for every $x\in \mathbb{T}^{N}$.
\end{definition}
%
%
%
We have
\begin{theorem}\label{Th_BEst}
	Let $P(x,D)$ be as in (\ref{HOR_Op}). Assume that $P(x,D)$ satisfies the (H.O.R)-condition. 
	Then there are $\varepsilon> 0$ and a positive constant $C$ such that 
	\begin{linenomath}
		\begin{align}\label{eq:B-Est}
		\|u \|_{\varepsilon}^{2}+ \sum_{k=1}^{N} \left( \| P^{k} u\|^{2}_{0} + \|P_{k}u\|_{-1}^{2}\right)
		\leq C \left( \sum_{m=0}^{N}\langle E_{m}P u, E_{m}u\rangle + \|u\|_{0}^{2}\right), 
		\end{align} 
	\end{linenomath}
for every $u \in C^{\infty}(\mathbb{T}^{N})$; $\|\cdot\|_{s}$ denotes the Sobolev norm of order $s$, 
 $ P^{k} $ and $ P_{k} $ as in \eqref{Pk_ud}, $E_{0}=1$ and $E_{m}=D_{m} \Lambda_{-1}$, $m=1,\dots, N$.
\end{theorem}
\begin{proof}
	Let $\displaystyle\bigcup_{\ell =1}^{M} \Omega_{j}$ be a finite open chart covering of $\mathbb{T}^{N}$
	and let $\displaystyle\sum_{\ell =1}^{M} \psi_{j}(x) =1$ be a subordinate partition of unit, $\psi_{j} \in C_{0}^{\infty}\left( \Omega_{j}; [0,1]\right)$.
	Let $u\in C^{\infty}(\mathbb{T}^{N})$, we set $u_{j} =\psi_{j}u$. We have
	 \begin{linenomath}
	 	\begin{equation*}
	 	\| u\|_{\varepsilon} = \| \sum_{j=1}^{M} u_{j}\|_{\varepsilon} \leq \sum_{j=1}^{M} \|  u_{j}\|_{\varepsilon}.
	 	\end{equation*}
	 \end{linenomath}
 By the result of Bolley, Camus and Nourrigat, Proposition 1.5 in \cite{BCN-82}, there is a positive $\varepsilon $, such that
 \begin{linenomath}
 	\begin{equation*}
 	\| u_{j}\|_{\varepsilon}  \leq  C_{j} \left( \sum_{\ell=1}^{2N} \| Q_{\ell} u_{j}\|_{0} + \| u_{j} \|_{\sigma_{j}} \right),
 	\end{equation*}
 \end{linenomath}
 where $Q_{\ell}$ are as in \eqref{Eq:Qj}.
 Since
 \begin{linenomath}
 	\begin{equation*}
 	\| u_{j} \|_{\sigma_{j}} \leq \tilde{C}_{j} \|u\|_{\sigma_{j}} \leq \delta_{j} \|u\|_{\varepsilon} + \widetilde{\widetilde{C}}_{j} \|u\|_{0},
 	\quad
 	\text{ for every } j,
 	\end{equation*}
 \end{linenomath}
 %
 where $\delta_{j}$ are suitable small positive numbers that we will choose later, and 
 \begin{linenomath}
 	\begin{equation*}
 	 \| Q_{\ell} u_{j}\|_{0} \leq \| \psi_{j} Q_{\ell} u\|_{0} + \| [Q_{\ell}, \psi_{j}] u\|_{0} \leq \|  Q_{\ell} u\|_{0} + B_{j} \| u\|_{0}. 
 	\end{equation*}
 \end{linenomath}
Then there are positive constants $A$ and $B$ such that
 \begin{linenomath}
 	\begin{equation*}
 	\| u\|_{\varepsilon} \leq A \sum_{\ell=1}^{2N} \| Q_{\ell} u\|_{0}  + B \|u\|_{0}  + C \left( \sum_{j=1}^{M}C_{j} \delta_{j}\right)\|u\|_{0}.
 	\end{equation*}
 \end{linenomath}
 We choose $\delta_{j} $ such that  $ \sum_{j=1}^{M}C_{j} \delta_{j}= \frac{1}{2} $; so, the last term on the right hand side can be absorbed on the left 
 hand side. Summing up there is a suitable positive constant  $C$ such that
 \begin{linenomath}
 	\begin{equation*}
 	\| u\|_{\varepsilon} \leq C\left(    \sum_{\ell=1}^{2N} \| Q_{\ell} u\|_{0}  + \|u\|_{0} \right)
 	= C\left( \sum_{k=1}^{N} \| P^{k} u\|_{0} +  \sum_{k=1}^{N} \| P_{k} u\|_{-1}  + \|u\|_{0} \right).
 	\end{equation*}
 \end{linenomath}
 By the above estimate and the Proposition \ref{P1} we obtain \eqref{eq:B-Est}.
\end{proof}
%
\bigskip
\subsection{A class of globally analytic Hypoelliptic operatorst  with non-negative quadratic form}~\par
\vskip-6mm
We consider in $\mathbb{T}^{N}=\mathbb{T}^{n}_{t}\times \mathbb{T}_{x}^{m}$  operators
of the form 
\begin{linenomath}
	\begin{multline}\label{Op_ORH_t}
	P(t,D_{t},D_{x})= \sum_{j,\ell =1}^{n} a_{j,\ell}(t)D_{t_{j}}D_{t_{\ell}} 
	+
	\sum_{j,\ell =1}^{m} a_{n+j,n+\ell}(t)D_{x_{j}}D_{x_{\ell}} 
		\\
	+
	2\sum_{j=1}^{n}\sum_{\ell=1}^{m} a_{j,n+\ell}(t)D_{t_{j}}D_{x_{\ell}} 
	+ i\sum_{j=1}^{n} b_{j}(t)D_{t_{j}} +i \sum_{\ell=1}^{m} c_{\ell}(t)D_{x_{\ell}}
	+d(t,x),
	\end{multline}    
\end{linenomath}
where $d(t,x)$ is analytic function and $a_{j,\ell}(t)$, $j,\ell \in\left\{1,2,\dots ,N \right\}$,
$b_{j}(t)$, $j=1,\dots, n$, and $c_{\ell}$, $\ell= 1,\dots,m$,  are real-valued real analytic functions;
the matrix $\mathbf{A}(t) = \left(a_{j,\ell}(t)\right)$, associated to the principal part of the operator,
is symmetric, $a_{j,\ell}(t)= a_{\ell,j}(t)$, and positive semidefinite, $\mathbf{A}(t) \geq 0$ on $\mathbb{T}^{N}$.\\
Moreover, we assume the following: 
\vspace{-2em}
\begin{enumerate}
	\item[\textbf{\textit{(A1)}}] The operator $P(t,D_{t},D_{x})$, \eqref{Op_ORH_t}, satisfies the (H.O.R.)-condition globally on $\mathbb{T}^{N}$
	(see Definition \ref{HOR-c}.) 
	\item[\textbf{\textit{(A2)}}] Let  $n_{1} < n$.  We set $t=(t', t'')$, with $t'\in\mathbb{T}^{n_{1}}$ and $t''\in \mathbb{T}^{n-n_{1}}$.
	The matrix associated to $P(t,D_{t},D_{x})$ has the form
	\begin{linenomath}
		\begin{align*}
		\mathbf{A}(t)=
		\sbox0{$\begin{matrix} D & D\end{matrix}$}
		\left[\begin{array}{c|c}
		\mathlarger{\mathbf{A}}_{1}(t') & \mathmakebox[\wd0]{\mathlarger{\mathbf{A}}_{3}(t)}\\
		\hline
		{^{\mathsf{T}}}\hspace{-0.1em}\mathlarger{\mathbf{A}}_{3}(t)  & \vphantom{\begin{matrix} D \\ D \end{matrix}} \mathlarger{\mathlarger{\mathbf{A}}}_{2}(t)
		\end{array}\right].
		\end{align*}
	\end{linenomath}
	Where $\mathbf{A}_{1}(t')$ is $n_{1}\times n_{1} $-matrix such that $\mathbf{A}_{1}(t')>0$;\\
	$\mathbf{A}_{3}(t)$ is $n_{1}\times (n-n_{1}+m )$-matrix of the form
	\begin{linenomath}
		\begin{align*}
		\mathbf{A}_{3}(t)=
		\sbox0{$\begin{matrix} D & D\end{matrix}$}
		\left[\begin{array}{c|c}
		 \mathmakebox[\wd0]{\bigzero} & \mathlarger{\mathbf{A}}_{32}(t) 
		\end{array}\right].
		\end{align*}
	\end{linenomath}
	where $\bigzero$ is $n_{1}\times (n-n_{1})$-matrix whose entries are all zero
	and  $\mathbf{A}_{32}(t)$  is $n_{1}\times m$-matrix; ${^{\mathsf{T}}}\hspace{-0.1em}\mathlarger{\mathbf{A}}_{3}(t)$
    denotes the transpose of  $\mathbf{A}_{3}(t)$;\\
    $\mathbf{A}_{2}(t)$ is $(n-n_{1}+m)\times (n-n_{1}+m) $-matrix of the following form
	\begin{linenomath}
		\begin{align*}
		\mathbf{A}_{2}(t)=
		\sbox0{$\begin{matrix} D & D\end{matrix}$}
		\left[\begin{array}{c|c}
		\mathlarger{\mathbf{A}}_{22} (t') & \mathmakebox[\wd0]{\mathlarger{\mathbf{A}}_{23}(t)} \\
		\hline
		\mathlarger{\mathbf{A}}_{32}(t)& \vphantom{\begin{matrix} D \\ D \end{matrix}} \mathlarger{\mathlarger{\mathbf{A}}}_{33}(t)
		\end{array}\right]
		\end{align*}
	\end{linenomath}
	where  $\mathbf{A}_{22}(t')$  is $(n-n_{1})\times (n-n_{1}) $- diagonal matrix,  $\mathbf{A}_{23}(t)$  is $(n-n_{1})\times m $-matrix,
	$\mathbf{A}_{32}(t) ={^{\mathsf{T}}}\hspace{-0.1em}\mathbf{A}_{32}(t) $  and
	$\mathbf{A}_{33}(t)$  is $m\times m $-matrix.\\ 
    The operator $P(t,D_{t},D_{x})$ has this form
	\begin{linenomath}
		\begin{multline}\label{Op_ORH_1}
		 \longsum[6]_{j,\ell =1}^{n_{1}} a_{j,\ell}(t')D_{t_{j}}D_{t_{\ell}} 
		+\longsum[8]_{j =n_{1}+1}^{n} a_{j,j}(t') D_{t_{j}}^{2}
		+\longsum[6]_{j,\ell =1}^{m} a_{n+j,n+\ell}(t)D_{x_{j}}D_{x_{\ell}} 
		\\
		+2\sum_{j =1}^{n} \sum_{\ell =1}^{m} a_{j,n+\ell} (t)D_{t_{j}} D_{x_{\ell}}
		+ i\sum_{j=1}^{n} b_{j}(t)D_{t_{j}} +i \sum_{\ell=1}^{m} c_{\ell}(t)D_{x_{\ell}}
		+d(t,x).
		\end{multline}    
	\end{linenomath}
	\item[\textbf{\textit{(A3)}}] For all $j_{1} \in \left\{ 1, \dots, n_{1} \right\}$,
	$j_{2} , \, \ell \in \left\{  n_{1}+1, \dots, n\right\}$,
	and $k\in \left\{  1, \dots, m\right\}$ there is a positive constant $C_{\star}$  such that
	%
	\begin{linenomath}
		\begin{equation}\label{A3_1}
		\begin{cases}
		|a_{j_{2},n+k}(t)| \leq C_{\star} \left(a_{\ell,\ell}(t')\right)^{2},\\
		 a_{n+k,n+k}(t) \leq C_{\star} \left( a_{\ell,\ell}(t')\right)^{2},\\
		|a_{j_{1},n+k}(t)| \leq C_{\star} a_{\ell,\ell}(t'),\\
		| b_{j_{2}}(t) | \leq C_{\star} a_{\ell,\ell}(t'),\\
		| c_{k}(t) | \leq C_{\star} a_{\ell,\ell}(t').
		\end{cases} 
		\end{equation}
	\end{linenomath}
	We also assume that
	\begin{linenomath}
		\begin{equation}\label{A3_2}
		\longsum[7]_{i=n_{1}+1}^{n} a_{i,i}(t') \leq C_{\star} a_{\ell,\ell}(t').
		\end{equation}
	\end{linenomath}
\end{enumerate}
\begin{remark}
	Since $\mathbf{A}(t) $ is positive semidefinite matrix it follow that
	\begin{linenomath}
		\begin{equation*}
		|a_{j,\ell}(t)| 
		\leq \sqrt{ a_{j,j}(t) a_{\ell,\ell}(t)  } \leq \frac{a_{j,j}(t) + a_{\ell,\ell}(t)}{2}, \quad \forall j,\ell \in\left\{1,2,\dots ,N \right\}.
		\end{equation*}
	\end{linenomath}
     Thus, the second line in the assumption \eqref{A3_1} combined with the above inequality implies that 
     \begin{linenomath}
     	\begin{equation*}
     |	a_{n+k_{1},n+k_{2}}(t) | \leq C_{\star} \left( a_{\ell,\ell}(t')\right)^{2},
     	\end{equation*}
     \end{linenomath}
     for every $k_{1},\, k_{2} \in \left\{  1, \dots, m\right\}$.
\end{remark}
We state now our main results
\begin{theorem}\label{Th_1}
	Let $P(t, D_{t}, D_{x})$ be a second order partial differential equations with non-negative
	characteristic form as in \eqref{Op_ORH_t} with analytic coefficients defined on the 
	torus $\mathbb{T}^{n+m}$.
	Assume that $P$ satisfies the conditions \textbf{(A1)} and \textbf{(A2)} with $n_{1}=n$ 
	above described\footnote{With no condition \textbf{(A3)}.}.
	Then $P$ is globally analytic hypoelliptic.
\end{theorem} 
\begin{theorem}\label{Th_2}
	Let $P(t, D_{t}, D_{x})$ be a second order partial differential equations with non-negative
	characteristic form as in \eqref{Op_ORH_t} with analytic coefficients defined on the 
	torus $\mathbb{T}^{n+m}$. Assume that $P$ satisfies the conditions \textbf{(A1)}, \textbf{(A2)} and \textbf{(A3)}
	above described. Then $P$ is globally analytic hypoelliptic.
\end{theorem} 
\begin{remark}
	The class of second order partial differential equations with non-negative
	characteristic form described by the assumptions  \textbf{(A1)} and \textbf{(A2)}  with $n_{1}=n$
	contains a subclass of sums of squares operators of H\"ormander type studied in \cite{ch94};
	the class described by the assumptions \textbf{(A1)}, \textbf{(A2)} and \textbf{(A3)}
	contains a subclass of sum of squares operators of H\"ormander type studied in \cite{BC-2022}.
	We point out that not all the operators in the class under investigation
	can be written in the form of H\"ormander's operators.
	These claims are motivated by some examples shown in the appendix.
\end{remark} 
\begin{remark}
	Using the some strategy adopted to proof  Theorem \ref{Th_2} can be showed that:\\
	$\mathbf{i)}$ let $P$ be of the following form:
	\begin{linenomath}
		\begin{multline*}
		\longsum[6]_{j,\ell =1}^{n_{1}} a_{j,\ell}(t')D_{t_{j}}D_{t_{\ell}} 
		+\longsum[8]_{j =n_{1}+1}^{n} a_{j,j}(t') D_{t_{j}}^{2}
		+\longsum[6]_{j,\ell =1}^{m} a_{n+j,n+\ell}(t')D_{x_{j}}D_{x_{\ell}} 
		\\
		+2\sum_{j =1}^{n} \sum_{\ell =1}^{m} a_{j,n+\ell} (t')D_{t_{j}} D_{x_{\ell}}
		+ i\sum_{j=1}^{n_{1}} b_{j}(t)D_{t_{j}} +  i\longsum[7]_{j=n_{1}+1}^{n} b_{j}(t')D_{t_{j}} 
		\\
		+i \sum_{\ell=1}^{m} c_{\ell}(t')D_{x_{\ell}}
		+d(t,x).
		\end{multline*}    
	\end{linenomath}
	Assume that $P$ satisfies the assumptions $\mathbf{(A1)}$ and $\mathbf{(A2)}$. Then the operator $P$ is globally analytic hypoellitic. 
	In this case the assumptions $\eqref{A3_1}$ and $\eqref{A3_2}$ in $\mathbf{(A3)}$ are not needed.\\
	$\mathbf{ii)}$ let $P$ be of the following form:
	\begin{linenomath}
		\begin{multline*}
		\longsum[6]_{j,\ell =1}^{n_{1}} a_{j,\ell}(t')D_{t_{j}}D_{t_{\ell}} 
		+\longsum[8]_{j =n_{1}+1}^{n} a_{j,j}(t') D_{t_{j}}^{2}
		+\longsum[6]_{j,\ell =1}^{m} a_{n+j,n+\ell}(t)D_{x_{j}}D_{x_{\ell}} 
		\\
		+2\sum_{j =n_{1}+1}^{n} \sum_{\ell =1}^{m} a_{j,n+\ell} (t)D_{t_{j}} D_{x_{\ell}}
		+ i\sum_{j=1}^{n_{1}} b_{j}(t,x)D_{t_{j}} 
		+ i\sum_{j=n_{1}+1}^{n} b_{j}(t)D_{t_{j}} 
		\\
		+i \sum_{\ell=1}^{m} c_{\ell}(t)D_{x_{\ell}}
		+d(t,x).
		\end{multline*}    
	\end{linenomath}
	Assume that $P$ satisfies the assumptions $\mathbf{(A1)}$, $\mathbf{(A2)}$ 
	with $\mathbf{A}_{3}(t)\equiv \bigzero$, and $\mathbf{(A3)}$. 
	Then the operator $P$ is globally analytic hypoellitic. 
	In this case the regularity with respect to the $x$-variables can be obtained not using the induction argument
	as done in the proof of Theorem \ref{Th_1} but using the same strategy to obtain the regularity
	with respect to $t''$-variables  in the proof of Theorem \ref{Th_2} i.e.
	taking maximum advantage from the subelliptic estimate \ref{eq:B-Est_P}.
\end{remark}   
%
\section{Proof of Theorems \ref{Th_1} and \ref{Th_2}}
\renewcommand{\theequation}{\thesection.\arabic{equation}}
\setcounter{equation}{0} \setcounter{theorem}{0}
\setcounter{proposition}{0} \setcounter{lemma}{0}
\setcounter{corollary}{0} \setcounter{definition}{0}
\setcounter{remark}{1}
%
Preliminary remarks:
\begin{remark}\label{Rmk1}
	By the assumptions we have 
	\begin{linenomath}
		\begin{align}\label{eq:P_k1}
		&P^{k} = 2 \sum_{j=1}^{n_{1}} a_{j,k}(t') D_{t_{j}} + 2 \sum_{\ell=1}^{m} a_{n+\ell,k}(t) D_{x_{\ell}}, \quad k=1,\dots, n_{1};
		\\
		\nonumber
		&P^{k} = 2a_{k,k}(t') D_{t_{k}} + 2 \sum_{\ell=1}^{m} a_{n+\ell,k}(t) D_{x_{\ell}}, \quad k=n_{1}+1,\dots, n;
		\\
		\nonumber
		&P^{n+k} = 2 \sum_{\ell=1}^{m} a_{n+\ell,n+k}(t) D_{x_{\ell}}
		+ 2\sum_{j =1}^{n}  a_{j,n+k} (t)D_{t_{j}}, \quad k=1,\dots, m;
		\end{align}
	\end{linenomath} 
	and
	\begin{linenomath}
		\begin{align}\label{eq:P_k2}
		&P_{k} = \longsum[6]_{j,\ell =1}^{n_{1}} a_{j,\ell}^{(k)}(t')D_{t_{j}}D_{t_{\ell}} 
		+\sum_{j =1}^{n} a_{j,j}^{(k)}(t') D_{t_{j}}^{2}
		+\longsum[6]_{j,\ell =1}^{m} a_{n+j,n+\ell}^{(k)}(t)D_{x_{j}}D_{x_{\ell}} 
		\\
		\nonumber
		&
		\qquad\qquad\qquad\qquad\qquad\,\,\,+2\sum_{j =1}^{n} \sum_{\ell =1}^{m} a_{j,n+\ell}^{(k)} (t)D_{t_{j}} D_{x_{\ell}}, \quad k=1,\dots, n_{1};
		\\
		\nonumber
		&P_{k} =2\sum_{j =1}^{n} \sum_{\ell =1}^{m} a_{j,n+\ell}^{(k)} (t)D_{t_{j}} D_{x_{\ell}}
		+\longsum[6]_{j,\ell =1}^{m} a_{n+j,n+\ell}^{(k)}(t)D_{x_{j}}D_{x_{\ell}}, 
		\\
		\nonumber
		&\qquad\qquad\qquad\qquad\qquad\qquad\qquad\qquad\qquad\qquad\qquad k=n_{1}+1,\dots, n;
		\\
		\nonumber
		&P_{n+k} = 0, \quad k=1,\dots, m.
		\end{align}
	\end{linenomath} 
	Due the hypothesis $\mathbf{A}_{1}(t')>0$, assumption \textbf{(A2)}, the vector fields 
	\begin{linenomath}
		\begin{equation}\label{eq:Pkt}
		\widetilde{P}^{k}= 2\sum_{j=1}^{n_{1}} a_{j,k}(t') D_{t_{j}} , \qquad k=1,\dots,n_{1},
		\end{equation}
	\end{linenomath}
are linearly independent for every $t'$ in $\mathbb{T}^{n_{1}}$.\\
The estimate obtained in the Theorem \ref{Th_BEst} adapted to our case becomes 
\begin{linenomath}
	\begin{align}\label{eq:B-Est_P}
	\|u \|_{\varepsilon}^{2}+ \sum_{k=1}^{N} \| P^{k} u\|^{2}_{0} +\sum_{k=1}^{n} \|P_{k}u\|_{-1}^{2}
	\leq C \left( \sum_{\ell=0}^{N}\langle E_{\ell}P u, E_{\ell}u\rangle + \|u\|_{0}^{2}\right), 
	\end{align} 
\end{linenomath}
where $ u \in C^{\infty}(\mathbb{T}^{N})$, $P^{k}$, $k=1,\dots, N$,
and $P_{k}$, $k=1,\dots,n$, are as in \eqref{eq:P_k1}
and \eqref{eq:P_k1} respectively, $E_{0}=1$, $E_{\ell}= D_{t_{\ell}} \Lambda^{-1}$,
$\ell=1,\dots,n$, and $E_{\ell}= D_{x_{\ell-n}} \Lambda^{-1}$, $\ell=n+1,\dots,n+m=N$,
here $\Lambda^{-1}$ denotes the pseudodifferential operator associate
to the symbol $(1+|\tau|^{2} + |\xi|^{2})^{-1/2}$.
\end{remark}
\begin{remark}\label{Rmk2}
	Let $P(t,D_{t},D_{x})$ be as in \eqref{Op_ORH_1} and we consider the problem
	$ Pu=f$, $f\in C^{\omega}(\mathbb{T}^{N})$.
	Due to the assumptions  \textbf{(A2)} and \textbf{(A3)}, the characteristic set of $P(t,D_{t},D_{x})$
	is contained in
	\begin{linenomath}
		\begin{equation*}
		\lbrace (t',t'',x, \tau', \tau'', \xi)\in T^{*}\mathbb{T}^{N}\setminus\{0\}\,:\, \tau'=0 \text{ and } |\tau''|+|\xi| >0\rbrace.  
		\end{equation*}
	\end{linenomath}
	So, the points $(t',t'',x, \tau', \tau'', \xi)\in T^{*}\mathbb{T}^{N}\setminus\{0\}$ with $\tau'\neq 0$ do not belong to $WF_{a}(u)$,
	the analytic wave-front set of $u$.\\
	In the special case $n_{1}=n$, the characteristic set of $P(t,D_{t},D_{x})$ is contained in 
	$\lbrace (t,x, \tau, \xi)\in T^{*}\mathbb{T}^{N}\setminus\{0\}\,:\, \tau=0 \text{ and } |\xi| >0\rbrace$, 
	the points $(t,x, \tau, \xi)\in T^{*}\mathbb{T}^{N}\setminus\{0\}$ with $\tau\neq 0$ do not belong to $WF_{a}(u)$. 
\end{remark}
To proof the results stated in the Theorems 
and in order make the paper as self-contained as possible,
we recall some results that will often used in the proof.\\
Let us first recall the notion of \textit{global Gevrey vector of order $s$}.
Let $P(y,D)$ be a real-analytic, linear partial differential operator
on $\mathbb{T}^{N}$ of order $m\geq 1$.
$v \in \mathscr{D}'(\mathbb{T}^{N})$ belongs to 
$ G^{s}(\mathbb{T}^{N}; P)$, the set of global Gevrey vectors
of order $s$ for $P(y,D)$, if $(P)^{j}v$
\footnote{Here $(P)^{j}v$ denotes the power $j$-th of the operator $P$.}
belongs to $L^{2}(\mathbb{T}^{N})$
and there is a positive constant $C$ such that
\begin{linenomath}
	\begin{equation*}
	\|(P)^{j}v \|_{0} \leq C^{j+1} j!^{sm},
	\end{equation*}
\end{linenomath} 
for every $j\in \mathbb{Z}_{+}$.\\
We recall that $G^{s}(\mathbb{T}^{N}) \subset G^{s}(\mathbb{T}^{N}; P)$
for every $P(y,D)$ and every $s$.

\begin{proposition}[\cite{BCM-78}]\label{Pr_BCM}
	Let $v \in G^{s}(\mathbb{T}^{N}; P)$, $P= P(y,D)$ as above.
	If $(y,\eta) \in T^{*}\mathbb{T}^{N}\setminus\{0\}$ does not belong to the characteristic variety
	of $P$, then $(y,\eta) \notin WF_{s}(v) $, the s-Gevrey wave front set of $u$.\\
	In particular, if $P$ is elliptic,  then $G^{s}(\mathbb{T}^{N})= G^{s}(\mathbb{T}^{N}, P)$.  
\end{proposition}
\begin{lemma}\label{L_BCCJ}[\cite{bccj_2016}]
	Let $f,g \in C^{\infty}(\mathbb{T}^{N})$ and $k\in \mathbb{Z}_{+}$, then
	\begin{linenomath}
		\begin{equation}\label{Lfg}
		\Delta^{k}(fg) = \sum_{\ell=0}^{k} 2^{\ell} \binom{k}{\ell} \sum_{j=0}^{k-\ell} \binom{k-\ell}{j} 
		\longsum[5]_{|I|=\ell} \left( D_{I} \Delta^{j}f\right)\left( D_{I} \Delta^{k-\ell-j}g\right),
		\end{equation}
	\end{linenomath}
	where  $I$ is a sequence $I=(i_{1}, \dots,i_{\ell}) $ of elements in $\{1,\dots, N\} $,
	$|I|\doteq \ell$ is the length of the sequence $I$ and $D_{I}= D_{y_{i_{1}}}\cdots D_{y_{i_{\ell}}}$.
\end{lemma}
\begin{remark}\label{Rk_L_BCCJ}
	Let $f,g \in C^{\infty}(\mathbb{T}^{N})$ and $k\in \mathbb{Z}_{+}$, then
	\begin{linenomath}
		\begin{multline}\label{eq:Rk_L3}
		[\Delta^{k}, f] g =
		\sum_{\ell=1}^{k} 2^{\ell} \binom{k}{\ell} \sum_{j=0}^{k-\ell} \binom{k-\ell}{j} 
		\longsum[5]_{|I|=\ell} \left( D_{I} \Delta^{j}f\right)\left( D_{I} \Delta^{k-\ell-j}g\right)
		\\
		+ \sum_{j=1}^{k} \binom{k}{j} \Delta^{j} f \Delta^{k-j} g.
		\end{multline}
	\end{linenomath}
\end{remark}
\begin{lemma}\label{L_BCCJ_2}[\cite{bccj_2016}]
	Let $f \in C^{\infty}(\mathbb{T}^{N})$ and $k\in \mathbb{Z}_{+}$. Assume that
	there exists a positive constant $C$, such that 
	\begin{linenomath}
		\begin{equation*}
		\| \Delta^{q} f\|_{0}
		\leq C^{q+1} q!^{2} \, \text{ if }\, q\leq k - 1.
		\end{equation*}
	\end{linenomath}
	Then,
	\begin{linenomath}
		\begin{equation*}
     	 \|D_I  \Delta^{q-|I|} f\|\leq C^{q-|I|/2+1}q!(q-|I|)!
    	\end{equation*}
	\end{linenomath}
    is valid for all sequences $I=(i_1,\ldots,i_j)$ of elements
    in
    $\{1,\ldots,m\}$, with length $|I|\doteq j \leq q < k$. Moreover, if $0 < |I| \leq k$, then
    \begin{linenomath}
    	\begin{equation*}
     	 \|D_I  \Delta^{k - |I|} f\|\leq C^{k-|I|/2}(k-1)!(k - |I| + 1)!
   	 \end{equation*}
\end{linenomath}	
$D_{I}$ is as in the previous lemma. 
\end{lemma}
%
\subsection{Proof of the Theorem \ref{Th_1}}
%
We consider the problem $Pu=f$, $f\in C^{\omega}(\mathbb{T}^{N})$. 
In this case the operator $P$ has the form \eqref{Op_ORH_1} where $n_{1}=n$.
Due to the remark \ref{Rmk2}, in order to establish that $u$ is analytic, i.e.
that $WF_{a}(u)=\emptyset$,
it is enough to show that $u$ is a global analytic vector for $\Delta_{x}$,
$\Delta_{x}$ denotes the Laplace operator with respect the $x$-variable.\\
The argument is the same used in \cite{bccj_2016} and in \cite{bj_2024}, but we shall present it here for the sake of completeness. 
We use a variation of the estimate \eqref{eq:B-Est_P}:
\begin{linenomath}
\begin{equation}
\label{Sub_1_T2}
	\|u \|_{\varepsilon}
	\leq C \left( \| P u\|_{0} + \|u\|_{0}\right), 
\end{equation}
\end{linenomath}
where $C$ is a suitable new positive constant.\\
We use the following interpolation inequality
which is valid for $\delta>0$ and $r$, $s$ and $t$ in $\mathbb{R}$ such that $r<s<t$:
\begin{linenomath}
	\begin{equation}\label{Inter_Est}
	\| v\|_{s} \leq  \delta \| v\|_{t} + \delta^{-\frac{s-r}{t-s}} \|v\|_{r}, \quad v \in H^{t}(\mathbb{T}^{N}).
	\end{equation}
\end{linenomath}
Setting $v=u$, $s=0$, $t=\varepsilon$, $r=-2$, and $\delta = (2C)^{-1}$,
the estimate \eqref{Sub_1_T2} can be rewritten as follow
\begin{linenomath}
	\begin{equation}\label{Sub_2_T2}
	\|u \|_{\varepsilon}
	\leq C \left(\| P u\|_{0} + \|u\|_{-2} \right), 
	\end{equation}
\end{linenomath}
for a different constant $C > 0$. 
We shall proceed by induction, so assume that there is a positive constant $A$ such that 
\begin{linenomath}
	\begin{equation}\label{Ind_x}
	\| \Delta_{x}^{j} u\|_{0} \leq A^{j+1}  j!^{2},  
	\end{equation}
\end{linenomath}
for all $j\leq p-1$.
We have to show that the above estimate holds for $j=p$.
We replace $u$ by $\Delta_{x}^{p}u$ in \eqref{Sub_2_T2}:
\begin{linenomath}
	\begin{equation}\label{eq:B-Est_P_x-2}
	\|\Delta^{p}_x u\|_{\varepsilon} 
\leq C \left( \|\Delta_x^{p} Pu \|_{0} + \| [\Delta_x^{p}, d ] u \|_{0} + \| \Delta_x^{p} u \|_{-2} \right).
	\end{equation} 
\end{linenomath}
 We select a
constant $B>0$ such that
\begin{equation*}
\|D_{I} \Delta_x^{j} Pu  \|_0 \leq B^{j + |I|/2 + 1} j!^{2} |I|!, \qquad \|D_I \Delta_x^j d \|_{L^\infty} \leq B^{j + |I|/2 + 1} j!^2 |I|!
\end{equation*}
for all $ j \in \mathbb{Z}_{+}$ and $ I = (i_{1},\, \dots, i_{|I|})$.
Now in view of \eqref{eq:Rk_L3} we obtain
\begin{align*}
\| [ \Delta_x^{p} , d ] u \|_{0}
& \leq \sum_{\ell =1}^{p} 2^{\ell} \binom{p}{\ell} \sum_{j=0}^{p-\ell} \binom{p-\ell}{j} \sum_{|I|=\ell} 
\| D_I \Delta_x^{j} d\|_{\infty} \| D_I \Delta_x^{p-\ell-j} u \|_{0} \\
& \quad
+ \sum_{ j = 1 }^{p} \binom{p}{ j} \| \Delta_x^{j} d \|_{\infty} \| \Delta_x^{ p-j } u \|_{0} \doteq S_{1} +S_{2}.
\end{align*}
Firstly we observe that by the inductive hypothesis we can estimate
\begin{equation*} 
S_{2} \leq  \sum_{j=1}^{p} \binom{p}{ j} B^{j+1} j!^{2} A^{p - j + 1 } ( p - j )!^2 
\leq
BA^{p+1} p!^{2} \sum_{j=1}^{p}  \left( \frac{B}{A} \right)^{j}.
\end{equation*}
For the other term we use again the inductive hypothesis in
conjunction with Lemma \ref{L_BCCJ_2}:
\begin{align*}
S_{1} \leq & 
 \sum_{\ell=1}^{p} 2^{\ell} \binom{p}{\ell} \sum_{j=1}^{p-\ell} \binom{p-\ell}{ j} 
\sum_{|I| =\ell} B^{j+\frac{\ell}{2}+1} j!^{2} \ell! A^{p-\frac{\ell}{2}-j + 1} (p-j)!(p-j-\ell)! \\
& + \sum_{\ell=1}^{p} 2^{\ell} \binom{p}{\ell}
\sum_{|I| =\ell} B^{\frac{\ell}{2}+1} \ell! A^{p-\frac{\ell}{2}} (p-1)!(p-\ell + 1)! \\
& 
\leq BA^{p+1} p!^2 \sum_{\ell=1}^p 2^\ell m^\ell \sum_{j=0}^{p-\ell}\left(\frac{B}{A}\right)^{j+\ell/2}.
\end{align*}
Finally, if $\gamma>0$ is such that $\|\Delta_x v\|_{-2}\leq \gamma
\|v\|$, $v\in C^\infty(\mathbb{T}^{n})$, gathering together in (\eqref{eq:B-Est_P_x-2})  all the information
obtained, and assuming $B\ll A$, we get
 \begin{eqnarray*}
\|\Delta_x^{p} u \|&\leq& CA^{p+1}p!^2
 \left(\frac{B^{p+1}}{A^{p+1}} + B\sum_{j=1}^p  \left( \frac{B}{A} \right)^{j} + 
 \tilde{A}\sum_{\ell=1}^p 2^\ell m^\ell \sum_{j=0}^{p-\ell}\left(\frac{B}{A}\right)^{j+\ell/2}
+ \frac{\gamma}{Ap^2}\right)\\
&\leq & CA^{p+1}p!^2\left( \frac{B}{A} + \frac{B^2}{A-B} + \frac{8mB}{A^{1/2}} + \frac{\gamma}{A} \right).
\end{eqnarray*}
If we choose $A>0$ so large such that the term between parentheses is
$\leq 1/C$  our 
inductive argument proves that  
\eqref{Ind_x} indeed holds for $j \leq p$, and every $p$, so $u$ is an analytic-vector for $\Delta_x$.  By the Proposition \ref{Pr_BCM}
we obtain that the analytic wave front set of $u$ is empty. We conclude that $u$ is analytic.\\
%
\subsection{Proof of the Theorem \ref{Th_2}}
%
We consider the problem $Pu=f$, $f\in C^{\omega}(\mathbb{T}^{N})$. 
We recall that in this case $P$ satisfies the assumptions $\mathbf{(A1)}$,
$\mathbf{(A2)}$, with $n_{1} < n$, and $\mathbf{(A3)}$, and has the following form: 
\begin{linenomath}
	\begin{multline}\label{Op_Th2}
	P(t,D_{t},D_{x})=
	\longsum[6]_{j,p =1}^{n_{1}} a_{j,p}(t')D_{t_{j}}D_{t_{p}} 
	+\longsum[8]_{j =n_{1}+1}^{n} a_{j,j}(t') D_{t_{j}}^{2}
	\\
	+\longsum[6]_{j,p =1}^{m} a_{n+j,n+p}(t)D_{x_{j}}D_{x_{\ell}} 
	+2\sum_{j =1}^{n} \sum_{p =1}^{m} a_{j,n+p} (t)D_{t_{j}} D_{x_{p}}
	+ i\sum_{j=1}^{n} b_{j}(t)D_{t_{j}} 
	\\
	+i \sum_{p=1}^{m} c_{p}(t)D_{x_{p}}
	+d(t,x).
	\end{multline}    
\end{linenomath}
Following the same strategy used in the proof of Theorem \ref{Th_1}, we have
\begin{linenomath}
	\begin{equation}\label{eq:Dx_Est}
	\|\Delta_{x}^{p}u \|_{0} \leq C_{1}^{p+1} p!^{2},
	\end{equation}
\end{linenomath}
where $C_{1}$ is suitable large constant.\\
By the above estimate, the Proposition \ref{Pr_BCM} and the Remark \ref{Rmk2}
we observe that only the points of the form $(t', t'',x, 0,\tau'', 0)$ may belong to $WF_{a}(u)$.\\
Thus, in order to show that $u \in C^{\omega}(\mathbb{T}^{N})$,
it is enough to show that there is a positive constant $B$ such that
\begin{linenomath}
	\begin{equation*}
	\| D_{t_{r}}^{q} u\|_{0} \leq B^{q+1}  q!, \qquad \forall q\in \mathbb{Z}_{+} \text{ and } \, r\in \{ n_{1}+1, \dots, n\}.
	\end{equation*}
\end{linenomath}
To do it, the main tool will be the subelliptic estimate.
We replace $u$ by $D^{q}_{t_{r}} u$ in \eqref{eq:B-Est}:
\begin{linenomath}
	\begin{multline}\label{eq:Est_Dt}
	\|D^{q}_{t_{r}} u \|_{\varepsilon}^{2}+ \sum_{j=1}^{N} \| P^{j} D^{q}_{t_{r}} u\|^{2}_{0} +  \sum_{j=1}^{n}\|P_{j}D^{q}_{t_{r}} u\|_{-1}^{2}
	\\
	\leq C \left( \sum_{\ell=0}^{N} \langle E_{\ell}P D^{q}_{t_{r}} u, E_{\ell} D^{q}_{t_{r}} u\rangle + \| D^{q}_{t_{r}} u\|_{0}^{2}\right), 
	\end{multline} 
\end{linenomath}
where we recall that $P^{k}$, $k=1,\dots,N$ and $P_{k}$, $k=1,\dots,n$, are as in \eqref{eq:P_k1}
and \eqref{eq:P_k1} respectively, $E_{0}=1$, $E_{\ell}= D_{t_{\ell}} \Lambda^{-1}$,
$\ell=1,\dots,n$, and $E_{\ell}= D_{x_{\ell-n}} \Lambda^{-1}$, $\ell=n+1,\dots,n+m=N$,
here $\Lambda^{-1}$ denotes the pseudodifferential operator associate
to the symbol $(1+|\tau|^{2} + |\xi|^{2})^{-1/2}$.\\
We begin to estimate the last term on the right hand side.
By the Parseval-Plancherel identity and Young's inequality for products, we have
\begin{linenomath}
	\begin{multline}\label{eq:Est_Rm}
	\| D^{q}_{t_{r}} u\|_{0}^{2} = (2\pi)^{N}\sum_{(\tau,\xi) \in \mathbb{Z}^{N}} |\tau_{r}^{q} \, \widehat{u}(\tau,\xi)|^{2}
	\\
	= (2\pi)^{N} \sum_{\substack{ (\tau,\xi) \in \mathbb{Z}^{N} \\ \|(\tau,\xi)\|\leq q }} |\tau_{r}|^{2q} \, |\widehat{u}(\tau,\xi)|^{2}
	+
	(2\pi)^{N} \sum_{\substack{ (\tau,\xi) \in \mathbb{Z}^{N} \\ \|(\tau,\xi)\| > q }} |\tau_{r}|^{2q} \, |\widehat{u}(\tau,\xi)|^{2}
	\\
	\leq q^{2q} \| u\|_{0}^{2} 
	+ (2\pi)^{N} \sum_{\substack{ (\tau,\xi) \in \mathbb{Z}^{N} \\ \|(\tau,\xi)\| > q }} \left( \delta |\tau_{r}|^{2(q+\varepsilon)}  + \delta^{-\frac{q}{\varepsilon}}\right)
	\, |\widehat{u}(\tau,\xi)|^{2}
	\\
	\leq
	C_{\delta}^{2(q+1)} q^{2q} + \delta \| D^{q}_{t_{r}} u\|_{\varepsilon}^{2},
	\end{multline}
\end{linenomath}
where $\delta$ is a positive real number. Choosing $\delta < (42(N+1)C)^{-1} $, where
$C$ is the constant on the right hand side of \eqref{eq:Est_Dt}, we conclude 
that $ C \| D^{q}_{t_{r}} u\|_{0}^{2}$ can be estimate estimate by a term that gives analytic growth
plus a term that can be absorbed on the right hand side of \eqref{eq:Est_Dt}.\\  
Consider now the first term on the right hand side of \eqref{eq:Est_Dt}. Let $C_{1}$ be a positive constant
such that $\| E_{\ell} \|_{L^{2} \rightarrow L^{2}} \leq C_{1}$.  We have
\begin{linenomath}
	\begin{multline}\label{eq:Est_T_S0}
	|\sum_{\ell=0}^{N} \langle E_{\ell}P D^{q}_{t_{r}} u, E_{\ell} D^{q}_{t_{r}} u\rangle |
	\leq
	\sum_{\ell=0}^{N} |\langle E_{\ell} [P, D^{q}_{t_{r}}] u, E_{\ell} D^{q}_{t_{r}} u\rangle |
	\\
	\hspace{16em}+
	\sum_{\ell=0}^{N} | \langle E_{\ell} D^{q}_{t_{r}}  P u, E_{\ell} D^{q}_{t_{r}} u\rangle |
	\\
	\leq
	\sum_{\ell=0}^{N} |\langle E_{\ell} [P, D^{q}_{t_{r}}] u, E_{\ell} D^{q}_{t_{r}} u\rangle |
	+ C_{1}^{4} (N+1)^{2} \| D^{q}_{t_{r}} f\|_{0}^{2} +  \|D^{q}_{t_{r}} u\|_{0}^{2}
	\\
	\leq
	\sum_{\ell=0}^{N} |\langle E_{\ell} [P, D^{q}_{t_{r}}] u, E_{\ell} D^{q}_{t_{r}} u\rangle |
	+ C_{2}^{2(q+1)} q!^{2} + \|D^{q}_{t_{r}} u\|_{0}^{2},
	\end{multline}
\end{linenomath}
where $C_{2}$ is a suitable positive constant.
Using \eqref{eq:Est_Rm}, the last term gives an analytic growth rate
plus a term that can be absorbed on the left side of \eqref{eq:Est_Dt}.\\
Let us examine a single term in the sum.
Since $r \in \{n_{1}+1, \dots, n \}$, we have
 \begin{linenomath}
 	\begin{multline}\label{eq:Est_T_S1}
 	|\langle  E_{\ell} [P, D^{q}_{t_{r}}] u,  E_{\ell} D^{q}_{t_{r}} u\rangle |
 	\\
 	\leq 
 	\sum_{\nu =1}^{q} \longsum[7]_{j,p = 1}^{m}  \binom{q}{\nu} 
 	|\langle E_{\ell} \left( D_{t_{r}}^{\nu} a_{n+j,n+p}(t)\right)  D^{q-\nu}_{t_{r}} D_{x_{j}} D_{x_{p}} u, E_{\ell} D^{q}_{t_{r}} u\rangle |
 	\\
 	+2 \sum_{\nu =1}^{q}
 	\sum_{j =1}^{n_{1}} \sum_{p =1}^{m} \binom{q}{\nu} 
 	|\langle E_{\ell} \left( D_{t_{r}}^{\nu} a_{j,n+p} (t) \right)  D^{q-\nu}_{t_{r}} D_{t_{j}} D_{x_{p}} u, E_{\ell}D^{q}_{t_{r}} u\rangle |
 	\\
 	+
 	2 \sum_{\nu =1}^{q}
 	\longsum[7]_{j =n_{1}+1}^{n} \sum_{p =1}^{m} \binom{q}{\nu} 
 	|\langle E_{\ell} \left( D_{t_{r}}^{\nu} a_{j,n+p} (t) \right)  D^{q-\nu}_{t_{r}} D_{t_{j}} D_{x_{p}} u,  E_{\ell} D^{q}_{t_{r}} u\rangle |
 	\\
 	+ 
 	\sum_{\nu =1}^{q} \longsum[7]_{j=n_{1}+1 }^{n} 
 	\binom{q}{\nu} 
 	|\langle E_{\ell} \left( D_{t_{r}}^{\nu} b_{j}(t) \right)  D^{q-\nu}_{t_{r}} D_{t_{j}} u, E_{\ell} D^{q}_{t_{r}} u\rangle |
 	\\
 	 +
 	 \sum_{\nu =1}^{q} \sum_{p=1}^{m}
 	 \binom{q}{\nu} 
 	 |\langle E_{\ell} \left( D_{t_{r}}^{\nu} c_{p}(t)\right) D^{q-\nu}_{t_{r}} D_{x_{p}}u, E_{\ell} D^{q}_{t_{r}} u\rangle |
 	 \\
 	 +
 	 \sum_{\nu =1}^{q} \sum_{j=1}^{n_{1}} 
 	 \binom{q}{\nu} 
 	 |\langle  E_{\ell} \left( D_{t_{r}}^{\nu} b_{j}(t) \right)  D^{q-\nu}_{t_{r}} D_{t_{j}} u, E_{\ell} D^{q}_{t_{r}} u\rangle |
 	 \\
 	 +
 	 \sum_{\nu =1}^{q}
 	 \binom{q}{\nu} 
 	 |\langle E_{\ell} \left( D_{t_{r}}^{\nu}   d(t,x)  \right) D^{q-\nu}_{t_{r}} u, E_{\ell} D^{q}_{t_{r}} u\rangle |
 	 =\sum_{i=1}^{7} J_{i}.
 	\end{multline}
 \end{linenomath}
Let us examine each terms separately.\\
\textbf{Term $J_{7}$}: we have
\begin{linenomath}
	\begin{multline*}
	J_{7} = \sum_{\nu =1}^{q}
	\binom{q}{\nu} 
	|\langle E_{\ell} \left( D_{t_{r}}^{\nu}   d(t,x)  \right) D^{q-\nu}_{t_{r}} u, E_{\ell} D^{q}_{t_{r}} u\rangle |
	\\
	\leq
	C_{1}^{2} 
	 \sum_{\nu =1}^{q}
	\binom{q}{\nu} 
	\|\left( D_{t_{r}}^{\nu}   d(t,x)  \right) D^{q-\nu}_{t_{r}} u\|_{0} \| D^{q}_{t_{r}} u\|_{0}
	\\
	\leq
	C_{1}^{2} C_{d}\sum_{\nu =1}^{q}
	\binom{q}{\nu}  C_{d}^{\nu} \nu!  \| D^{q-\nu}_{t_{r}} u\|_{0} \| D^{q}_{t_{r}} u\|_{0}
	\leq
	\sum_{\nu =1}^{q}
	 \widetilde{C}_{d}^{\nu} q^{\nu}  \| D^{q-\nu}_{t_{r}} u\|_{0} \| D^{q}_{t_{r}} u\|_{0}.
	\end{multline*}
\end{linenomath} 
By the Parseval-Plancherel identity and Young's inequality for products the factor
$\widetilde{C}_{d}^{\nu} q^{\nu}  \| D^{q-\nu}_{t_{r}} u\|_{0}  $   can be bounded by
$ 2^{-\nu} \| D^{q}_{t_{r}} u\|_{0} + \widetilde{C}_{d}^{q} q^{q}  \|  u\|_{0}$. 
We obtain
\begin{linenomath}
	\begin{multline}\label{Est:J8}
	J_{7} 
	\leq
	\left(
	\sum_{\nu =1}^{q} \left(\frac{1}{2}\right)^{\nu} \| D^{q}_{t_{r}} u\|_{0}^{2}
	+
	\sum_{\nu =1}^{q} 2^{q-\nu} \widetilde{C}_{d}^{q} q^{q}  \|  u\|_{0} \| D^{q}_{t_{r}} u\|_{0}
	\right)
	\\
	\leq
	2 \| D^{q}_{t_{r}} u\|_{0}^{2} +   \left( 2 \widetilde{C}_{d}\right)^{2q} q^{2q} \| u\|_{0}^{2}
	\\
	\leq
	2 \| D^{q}_{t_{r}} u\|_{0}^{2} + C_{3}^{2(q+1)}  q!^{2}, 
	\end{multline}
\end{linenomath} 
where $C_{3}$ is a suitable positive constants independent of $q$.
Here, we used that $\sum_{\nu =1}^{q} 2^{-\nu} \leq \sum_{\nu =0}^{\infty} 2^{-\nu} -1\leq 1$.
Using \eqref{eq:Est_Rm}, we conclude that $J_{7}$ can be estimate by
an analytic growth rate
plus a term that can be absorbed on the left side of \eqref{eq:Est_Dt}.

\textbf{Term $J_{6}$} on the right hand side of $\eqref{eq:Est_T_S1}$:
\begin{linenomath}
	\begin{multline}\label{eq:J_6-1}
	J_{6}=\sum_{\nu =1}^{q} \sum_{j=1}^{n_{1}} 
	\binom{q}{\nu} 
	|\langle  E_{\ell} \left( D_{t_{r}}^{\nu} b_{j}(t) \right)  D^{q-\nu}_{t_{r}} D_{t_{j}} u, E_{\ell} D^{q}_{t_{r}} u\rangle |
	\\
	\leq
	C_{1}^{2}
	\sum_{\nu =1}^{q} \sum_{j=1}^{n_{1}} 
	C_{b}^{\nu+1} q^{\nu}
	\| D_{t_{j}} D^{q-\nu}_{t_{r}} u\|_{0} \| D^{q}_{t_{r}} u\|_{0}
	\\
	\leq
	n_{1} C_{1}^{2} \sum_{\nu =1}^{q} \sum_{j=1}^{n_{1}} 2^{\nu} C_{b}^{2(\nu+1)} q^{2\nu}
	\| D_{t_{j}} D^{q-\nu}_{t_{r}} u\|_{0}^{2}
	+
	\sum_{\nu =1}^{q} 2^{-\nu}  \| D^{q}_{t_{r}} u\|_{0}
	\\
	\leq
	n_{1} C_{1}^{2} \sum_{\nu =1}^{q-1} \sum_{j=1}^{n_{1}}  (2C_{b})^{2(\nu+1)} q^{2\nu}
	\| D_{t_{j}} D^{q-\nu}_{t_{r}} u\|_{0}^{2}
	\\
	+
	n_{1} C_{1}^{2}  (2C_{b})^{2(q+1)} q^{2q}\sum_{j=1}^{n_{1}}  
	\| D_{t_{j}} u\|_{0}^{2}
	+ \| D^{q}_{t_{r}} u\|_{0},
	\end{multline}
\end{linenomath}
where $C_{b} $ is a suitable positive constant depending of the $b_{j}(t)$. 
Using \eqref{eq:Est_Rm}, the last two terms give analytic growth
plus a term that can be absorbed on the left side of \eqref{eq:Est_Dt}.
In order to estimate the terms in the first sum, $\| D_{t_{j}} D^{q-\nu}_{t_{r}} u\|_{0}$,
we take advantage by the assumptions $\mathbf{(A2)}$ and $\mathbf{(A3)}$.
More precisely, as pointed out in the Remark \ref{Rmk1}, we have that $\widetilde{P}^{k}$, \eqref{eq:Pkt},
$k=1,\dots, n_{1}$, are linearly independent, thus,  for every $j$, we have 
\begin{linenomath}
	\begin{multline}\label{Dtj_1n}
	D_{t_{j}} \!= \! \sum_{k=1}^{n_{1}}\! \text{\textcursive{p}}_{k,j}(t') \widetilde{P}^{k}(t',D_{t'})\!
	= \sum_{k=1}^{n_{1}} \! \text{\textcursive{p}}_{k,j}(t') \! \left( \!P^{k}(t',D_{t'}) \!-2\!  \sum_{p =1}^{m} a_{n+p,k}(t) D_{x_{p}} \!\right),
	\end{multline}
\end{linenomath}  
where we use that  $\widetilde{P}^{k} = P^{k} -2 \sum_{p =1}^{m} a_{n+p,k}(t) D_{x_{p}}$, $P^{k}$ as in \eqref{eq:P_k1}, $k=1,\dots, n_{1}$.
So, the first term on the right hand side of \eqref{eq:J_6-1} can be estimated by
\begin{linenomath}
	\begin{multline}\label{eq:J_6-2}
	n_{1}^{2} C_{1}^{2} 
	\sum_{\nu =1}^{q-1} \sum_{j=1}^{n_{1}} 
	\sum_{k=1}^{n_{1}}
	 (2C_{b})^{2(\nu+1)} q^{2\nu}
	\|\text{\textcursive{p}}_{k,j}(t') P^{k} D^{q-\nu}_{t_{r}} u\|_{0}^{2}
	\\
	+
	2n_{1}^{2}m C_{1}^{2} \sum_{\nu =1}^{q-1} \sum_{j=1}^{n_{1}} 
	\sum_{k=1}^{n_{1}}  \sum_{p =1}^{m}
	 (2C_{b})^{2(\nu+1)} q^{2\nu}
	\| \text{\textcursive{p}}_{k,j}(t') a_{n+p,k}(t) D_{x_{p}} D^{q-\nu}_{t_{r}} u\|_{0}^{2}
	\\
	\leq
	n_{1}^{3} C_{1}^{2} C_{ \text{\textcursive{p}} }
	\sum_{\nu =1}^{q-1} 
	\sum_{k=1}^{n_{1}}
	(2C_{b})^{2(\nu+1)} q^{2\nu}
	\| P^{k} D^{q-\nu}_{t_{r}} u\|_{0}^{2}
	\\
	+
	2n_{1}^{3}m C_{1}^{2} C_{ \text{\textcursive{p}} }\sum_{\nu =1}^{q-1}
	\sum_{k=1}^{n_{1}}  \sum_{p =1}^{m}
	(2C_{b})^{2(\nu+1)} q^{2\nu}
	\| a_{n+p,k}(t) D_{x_{p}} D^{q-\nu}_{t_{r}} u\|_{0}^{2},
	\end{multline}
\end{linenomath}
where $C_{ \text{\textcursive{p}} }= \displaystyle\sup_{k,j}\sup_{t'}| \text{\textcursive{p}}_{k,j}(t') | $.
We handle the two terms obtained on the right hand side separately. Concerning the first term,
we observe that $\nu$ $t_{r}$-derivatives were turned into a factor $C^{\nu+1} q^{\nu}$, $C$ suitable
positive constant independent of $q$ and $\nu$. Taking maximum advantage from the subelliptic estimate,
\eqref{eq:B-Est_P}, we can restart the process with $u $ replaced by $C^{\nu+1} q^{\nu}  D^{q-\nu}_{t_{r}} u$.
Iterating the process $(q-\nu)$-times, modulo terms which give analytic growth or terms 
of the same form as in the right hand side of \eqref{eq:Est_T_S1}, $J_{1},\dots, J_{5}$, which we will see how to deal with later,
we obtain the term
\begin{linenomath}
	\begin{equation*}
	C_{2}^{2(q+1)} q!^{2}
	\| P^{k}  u\|_{0}^{2},
	\end{equation*}
\end{linenomath}
i.e. analytic growth. Here $C_{2}$ is a suitable positive constant independent of $q$.\\
To handle the second term on the right hand side of \eqref{eq:J_6-2},
we take advantage by the assumption $\mathbf{(A3)}$,
\eqref{A3_1} and \eqref{A3_2}:
\begin{linenomath}
	\begin{equation*}
	|a_{n+p,k}(t)| = |a_{k,n+p}(t)|  \leq a_{k,k}(t') +a_{n+p,n+p}(t) \leq 2C_{\star} a_{r,r}(t'). 
	\end{equation*}
\end{linenomath}
Since $P^{r} = 2a_{r,r}(t') D_{t_{r}} + 2 \sum_{p=1}^{m} a_{n+p,r}(t) D_{x_{p}}$, we have
\begin{linenomath}
	\begin{multline*}
	 \sum_{\nu =1}^{q-1}
	\sum_{k=1}^{n_{1}}  \sum_{p =1}^{m}
	\widetilde{C}_{b}^{2(\nu+1)} q^{2\nu}
	\| a_{n+p,k}(t) D_{x_{p}} D^{q-\nu}_{t_{r}} u\|_{0}^{2}
	\\
	\hspace{-3em}
	\leq
	\sum_{\nu =1}^{q-1}
	\sum_{p =1}^{m}
	n_{1} 2^{2}C_{\star}^{2}\widetilde{C}_{b}^{2(\nu+1)} q^{2\nu}
	\| P^{r}  D^{q-\nu-1}_{t_{r}} D_{x_{p}} u\|_{0}^{2}
	\\
	+
	\sum_{\nu =1}^{q-2}
	\sum_{p_{1} =1}^{m} \sum_{p_{2} =1}^{m}
	n_{1}2^{2} C_{\star}^{2}\widetilde{C}_{b}^{2(\nu+1)} q^{2\nu}
	\| a_{n+p_{2},r}(t) D^{q-\nu-1}_{t_{r}} D_{x_{p_{2}}} D_{x_{p_{1}}} u\|_{0}^{2}
	\\
	+
	2^{2}C_{\star}^{2}\widetilde{C}_{b}^{2q} q^{2(q-1)}
	\sum_{p_{1} =1}^{m} \sum_{p_{2} =1}^{m}
	\| a_{n+p_{2},r}(t) D_{x_{p_{2}}} D_{x_{p_{1}}} u\|_{0}^{2}
	.
	\end{multline*}
\end{linenomath}
We remark that the last term gives analytic growth;
moreover, in the terms of the first sum, $\nu$ $t_{r}$-derivatives where turned
into a factor $q^{\nu}$ and one derivative it was turned into a $x$-derivative. 
Iterating the process, we obtain 
\begin{linenomath}
	\begin{multline}\label{eq:J_6-3}
	\sum_{\nu =1}^{q-1}
	\sum_{k=1}^{n_{1}}  \sum_{p =1}^{m}
	\widetilde{C}_{b}^{2(\nu+1)} q^{2\nu}
	\| a_{n+p,k}(t) D_{x_{p}} D^{q-\nu}_{t_{r}} u\|_{0}^{2}
	\\
	\hspace{-3em}
	\leq
	\sum_{i=1}^{q-1}\sum_{\nu =1}^{q-i}
	\sum_{p_{1} =1}^{m}\cdots \sum_{p_{i} =1}^{m}
	 C_{\star}^{2i}\widetilde{C}_{b}^{2(\nu+1)} q^{2\nu}
	\| P^{r}  D^{q-\nu-i}_{t_{r}} D_{x_{p_{1}}}\cdots D_{x_{p_{i}}}u\|_{0}^{2}
	\\
	+
	\sum_{i =1}^{q-1}
	\sum_{p_{1} =1}^{m}\cdots \sum_{p_{i+1} =1}^{m}
	C_{\star}^{2i}\widetilde{C}_{b}^{2(q-i+1)} q^{2(q-i)}
	\| D_{x_{p_{1}}} \cdots D_{x_{p_{i+1}}} u\|_{0}^{2}.
	\end{multline}
\end{linenomath}
Since we know that $u$ has analytic growth with respect $x$-derivatives,
we have 
\begin{linenomath}
	\begin{equation*}
	\sum_{i =1}^{q-1}
	\sum_{p_{1} =1}^{m}\cdots \sum_{p_{i+1} =1}^{m}
	C_{\star}^{2i}\widetilde{C}_{b}^{2(q-i+1)} q^{2(q-i)}
	\| D_{x_{p_{1}}} \cdots D_{x_{p_{i+1}}} u\|_{0}^{2}
	\leq C_{3}^{2(q+1)} q^{2q},
	\end{equation*}
\end{linenomath}
where $C_{3}$ is a suitable positive constant.\\
Concerning the terms in the first sum on the right hand side of \eqref{eq:J_6-3},
we can take maximum advantage by the subelliptic estimate, \eqref{eq:B-Est_P}, restarting
the process with $u$ replaced by  $C_{\star}^{i}\widetilde{C}_{b}^{\nu+1} q^{\nu}
D^{q-\nu-i}_{t_{r}} D_{x_{p_{1}}}\cdots D_{x_{p_{i}}}u$. After $s$-times,
modulo terms which give analytic growth or terms 
of the same form as in the right hand side of \eqref{eq:Est_T_S1}, $J_{1},\dots, J_{5}$, 
we will face of with terms of the form
\begin{linenomath}
	\begin{equation}\label{eq:J_6-iter_0}
	C_{\star}^{2\sum_{k=1}^{s}i_{k}}
	\widetilde{C}_{b}^{2(\sum_{k=1}^{s}(\nu_{k}+1))} q^{2\sum_{k=1}^{s}\nu_{k}}
	\| P^{r}  D^{q-\sum_{k=1}^{s}\nu_{k}-\sum_{k=1}^{s}i_{k}}_{t_{r}} D_{I_{1}}\cdots D_{I_{s}}u\|_{0}^{2},
	\end{equation}
\end{linenomath}
where $\nu_{k} \leq q-\nu_{k-1}-i_{k-1}$ and $D_{I_{k}}= D_{x_{p_{1}}}\cdots D_{x_{p_{i_{k}}}}$, $|I_{k}|=i_{k}$, $k=1,\dots,s$.
After a suitable number of steps so that  $q-\sum_{k=1}^{s}\nu_{k}-\sum_{k=1}^{s}i_{k}=0$,
we will obtain
 \begin{linenomath}
 	\begin{equation}\label{eq:J_6-iter}
 	C_{\star}^{2\sum_{k=1}^{s}i_{k}}
 	\widetilde{C}_{b}^{2( q-\sum_{k=1}^{s}i_{k}+1)} q^{2  q-\sum_{k=1}^{s}i_{k} }
 	\| P^{r} D_{I_{1}}\cdots D_{I_{s}}u\|_{0}^{2}.
 	\end{equation}
 \end{linenomath}
By \eqref{eq:Dx_Est} and since $\sum_{k=1}^{s}|I_{k}|=\sum_{k=1}^{s}i_{k}$,  we conclude that
also these terms give analytic growth.

\textbf{Term $J_{5}$} on the right hand side of \eqref{eq:Est_T_S1}:
\begin{linenomath}
	\begin{multline*}
	J_{5}=\sum_{\nu =1}^{q} \sum_{p=1}^{m}
	\binom{q}{\nu} 
	|\langle E_{\ell} \left( D_{t_{r}}^{\nu} c_{p}(t)\right) D^{q-\nu}_{t_{r}} D_{x_{p}}u, E_{\ell} D^{q}_{t_{r}} u\rangle |
	\\
	\leq
	C_{1}^{2} m \sum_{\nu =1}^{q} \sum_{p=1}^{m}
	\left(\frac{q!}{\nu!(q-\nu)!}\right)^{2} 2^{\nu}
	 \| \left( D_{t_{r}}^{\nu} c_{p}(t)\right) D^{q-\nu}_{t_{r}} D_{x_{p}}u\|_{0}^{2}  
	 +
	 \sum_{\nu =1}^{q}
	 \left( \frac{1}{2} \right)^{\nu}
	  \| D^{q}_{t_{r}} u\|_{0}^{2}
	  \\
	  \leq
	  C_{1}^{2} m \sum_{\nu =1}^{q-1} \sum_{p=1}^{m}
	  \left(\frac{q!}{\nu!(q-\nu)!}\right)^{2} 2^{\nu}
	  \| \left( D_{t_{r}}^{\nu} c_{p}(t)\right) D^{q-\nu}_{t_{r}} D_{x_{p}}u\|_{0}^{2}  
	  +
	  \| D^{q}_{t_{r}} u\|_{0}^{2}
	  \\
	  +
	  C_{1}^{2} m C_{c}^{2(q+1)} q!^{2}\sum_{p=1}^{m} \| D_{x_{p}}u\|_{0}^{2} .
	\end{multline*}
\end{linenomath}
The last term term gives analytic growth.
Using \eqref{eq:Est_Rm}, the second term gives analytic growth
plus a term that can be absorbed on the left side of \eqref{eq:Est_Dt}.
In order to estimate the terms in the sum
we take advantage from the assumption $\mathbf{(A3)}$, \eqref{A3_1}:
\begin{linenomath}
	\begin{equation*}
	| c_{p}(t) | \leq C_{\star} a_{r,r}(t').
	\end{equation*}
\end{linenomath}
We distinguish two cases. If $a_{r,r}(t') >0$, we are reduced to a situation analogous
to the one already studied to estimate $J_{6}$;
indeed the addends of the sum can be estimated as follows
\begin{linenomath}
	\begin{multline}\label{J5_ell}
	\left(\frac{q!}{\nu!(q-\nu)!}\right)^{2} 2^{\nu}
	\| \left( D_{t_{r}}^{\nu} c_{p}(t)\right) D^{q-\nu}_{t_{r}} D_{x_{p}}u\|_{0}^{2}
	\\
	\leq
	C_{c}^{2(\nu+1)} q^{2\nu} \left( \| P^{r}  D^{q-\nu-1}_{t_{r}} D_{x_{p}}u\|_{0}^{2}
	+
	\sum_{p_{2}=1}^{m} \| a_{n+p_{2},r}(t)  D^{q-\nu-1}_{t_{r}} D_{x_{p_{2}}}D_{x_{p}}\|_{0}^{2}\right).
	\end{multline}
\end{linenomath} 
We can use the same strategy adopted to handle the last term in \eqref{eq:J_6-2}, see also \eqref{eq:J_6-3}.
So, also in this case, modulo terms that give analytic growth, taking advantage from the subelliptic estimate
and iterating the process we will obtain again analytic growth,
of course modulo terms of the same form of $J_{1},\dots, J_{4}$, right hand side of \eqref{eq:Est_T_S1}.\\
Now, we assume that $a_{r,r}^{-1}(0) \neq \emptyset$. In order to discuss this case we use a partition
of unity in $\mathbb{T}^{n}$.
Let $\bigcup_{s =1}^{M} \Omega_{s}$ be a finite open chart covering of $\mathbb{T}^{n}$
and $\rho_{s}(t)$ in  $ C_{0}^{\infty}\left( \Omega_{s}; [0,1]\right)$ such that
$\sum_{s =1}^{M} \rho_{s}(t) =1$ be a subordinate partition of unit.\\
Using the partition of unit we have to estimate  $\| \rho_{s}(t) \left( D_{t_{r}}^{\nu} c_{p}(t)\right) D^{q-\nu}_{t_{r}} D_{x_{p}}u\|$
, $ s=1,\dots,M$.
If $\Omega_{s} \cap a_{r,r}^{-1}(0) = \emptyset $ we proceed as before, \eqref{J5_ell}.\\
We assume that $\Omega_{s} \cap a_{r,r}^{-1}(0) \neq \emptyset $.
Without loss of generality and for seek of simplicity we assume that the origin belongs to $\Omega_{s}$,
$ a_{r,r}(0)=0$ and $a_{r,r}(t_{1},0) $ has an isolated zero of multiplicity $k_{r}$ at zero.
We point out that the estimate that we will obtain is independent by $s$.\\
By the Weierstrass  preparation theorem (see for example \cite{Treves2022}) we may write
\begin{linenomath}
	\begin{equation}\label{eq:arr}
	\hspace*{-0.5em}
	a_{r,r}(t')=\widetilde{a}_{r,r}(t')\text{\textcursive{p}}_{r}(t_{1},\widetilde{t}\, ),	\text{ where }\quad 
	\text{\textcursive{p}}_{r}(t_{1},\widetilde{t}\,)= t_{1}^{k_{r}}+ \sum_{\nu=1}^{k_{r}} \text{\textcursive{p}}_{r,\nu}(\widetilde{t}\,)t_{1}^{\nu},
	\end{equation}
\end{linenomath}
where $ \widetilde{a}_{r,r}(t')$ is an analytic function on $\Omega_{s}$ nowhere vanishing
and $\text{\textcursive{p}}_{r,\nu}(\widetilde{t}\,)$
are analytic function vanishing at $\widetilde{t}=0$, $\widetilde{t}= (t_{2},\dots,t_{n_{1}})$.\\
We consider now $c_{p}(t_{1},0,t'')$; by the assumption $\mathbf{(A3)}$, \eqref{A3_1},
it vanishes at $t_{1}=0$ for every $t''$.
We may always perform a small linear change of the variables $ t' $ in such a
way that both $ a_{r,r}(t_{1}, 0) $ and  $ c_{p}(t_{1}, 0, t'') $
have an isolated zero at $ t_{1} = 0 $.
By the Weierstrass preparation theorem we may write
\begin{linenomath}
\begin{equation}\label{eq:cp_1}
c_{p}(t) = \widetilde{c}_{p}(t', t'')\, \text{\textcursive{q}}_{p}\!(t_{1}, \widetilde{t}, t''), 
\text{ where }
\text{\textcursive{q}}_{p}\!(t_{1}, \tilde{t}, t'') =  t_{1}^{h_{p}}\! + \! 
\sum_{\mu=1}^{h_{p}}      \text{\textcursive{q}}_{p,\mu}\!(t_{1}, \widetilde{t}, t'') t_{1}^{h_{p}-\mu}.
\end{equation} 
\end{linenomath}
By the assumption \eqref{A3_1}: $ | c_{p} (t)| \leq C_{\star} | a_{r,r}(t') | $; it follow that
$ |\text{\textcursive{q}}_{p} | \leq \widetilde{C}_{\star} | \text{\textcursive{p}}_{r}| $.
This implies that $ k_{r}\leq h_{p}$ and that $ \text{\textcursive{p}}_{r} $ divides $ \text{\textcursive{q}}_{p} $, 
i.e.
\begin{linenomath}
	\begin{equation}
	\label{eq:cp_2}
	c_{p}(t) = \widetilde{c}_{p}(t', t'') \text{\textcursive{p}}_{r} (t_{1}, \widetilde{t}\,)
	\widetilde{\text{\textcursive{q}}}_{p,r}(t_{1}, \widetilde{t}, t''),
	\end{equation}
\end{linenomath}
where $ \widetilde{\text{\textcursive{q}}}_{p,r}$ denotes a Weierstrass polynomial of degree $ h_{p}-k_{r} $.\\
We note explicitly that the dependence on $ t'' $ of the functions $c_{p} $ is confined to the factors
$\widetilde{c}_{p}  $ and $\widetilde{\text{\textcursive{q}}}_{p,r} $.
We have that
\begin{linenomath}
	\begin{multline}
	\label{eq:Dcp}
	D_{t_{r}}^{\nu} c_{p}(t) =
	\left[ D_{t_{r}}^{\nu}  (\widetilde{c}_{p}(t', t'') \widetilde{\text{\textcursive{q}}}_{p,r}(t_{1}, \widetilde{t}, t'') ) \right] 
	\text{\textcursive{p}}_{r} (t_{1}, \widetilde{t}\,)
	\\
	=
	\left[ D_{t_{r}}^{\nu}  (\widetilde{c}_{p}(t', t'') \widetilde{\text{\textcursive{q}}}_{p,r}(t_{1}, \widetilde{t}, t'') ) \right] 
	\left[\widetilde{a}_{r,r}(t')\right]^{-1} a_{r,r}(t').
	\end{multline}
\end{linenomath}
By the previous considerations we have
\begin{linenomath}
	\begin{multline}
	\left(\frac{q!}{\nu!(q-\nu)!}\right)^{2} 2^{\nu}
	\| \rho_{s}(t)\left( D_{t_{r}}^{\nu} c_{p}(t)\right) D^{q-\nu}_{t_{r}} D_{x_{p}}u\|_{0}^{2}
	\\
	=
	\frac{q!^{2} 2^{\nu}}{\nu!^{2}(q-\nu)!^{2}} 
	\| \rho_{s}(t)     \left( D_{t_{r}}^{\nu}  (\widetilde{c}_{p}(t) \widetilde{\text{\textcursive{q}}}_{\,p,r}(t) ) \right) 
	\left[\widetilde{a}_{r,r}(t')\right]^{-1} a_{r,r}(t') D_{t_{r}} D^{q-\nu-1}_{t_{r}} D_{x_{p}}u\|_{0}^{2}
	\\
	\leq
	\widetilde{C}_{c}^{2(\nu+1)} q^{2\nu} \left( \| P^{r}  D^{q-\nu-1}_{t_{r}} D_{x_{p}}u\|_{0}^{2}
	+
	\sum_{p_{2}=1}^{m} \| a_{n+p_{2},r}(t)  D^{q-\nu-1}_{t_{r}} D_{x_{p_{2}}}D_{x_{p}}\|_{0}^{2}\right),
	\end{multline}
\end{linenomath}
where we used that $P^{r} = 2a_{r,r}(t') D_{t_{r}} + 2 \sum_{p=1}^{m} a_{n+p,r}(t) D_{x_{p}}$, 
$\widetilde{C}_{c}$ is a suitable positive constant.
We obtain
\begin{linenomath}
	\begin{multline*}
	J_{5}
	\leq
	C_{1}^{2} m^{2} \sum_{\nu =1}^{q-1} 
	\widetilde{C}_{c}^{2(\nu+1)} q^{2\nu}  \| P^{r}  D^{q-\nu-1}_{t_{r}} D_{x_{p}}u\|_{0}^{2}
	\\
	+
	C_{1}^{2} m \sum_{\nu =1}^{q-1} 
	\sum_{p_{1}=1}^{m}
	\sum_{p_{2}=1}^{m} 
	\widetilde{C}_{c}^{2(\nu+1)} q^{2\nu}
	\| a_{n+p_{2},r}(t)  D^{q-\nu-1}_{t_{r}} D_{x_{p_{1}}} D_{x_{p_{1}}}\|_{0}^{2}
	\\
	+
	 C_{2}^{2(q+1)} q!^{2} \| D_{x_{p}}u\|_{0}^{2}
	+
	\| D^{q}_{t_{r}} u\|_{0}^{2}.
	\end{multline*}
\end{linenomath}
The third term gives analytic growth rate; using \eqref{eq:Est_Rm}, the last term gives analytic growth
plus a term that can be absorbed on the left side of \eqref{eq:Est_Dt}. 
About the first term on the right hand side,
we remark that $\nu+1$ $t_{r}$-derivatives were turned into a factor $q^{\nu}$ and 
one derivative in $x$.
Also in this case, we can use the same strategy adopted to handle
the last term in \eqref{eq:J_6-2}, see also \eqref{eq:J_6-3}.
So, taking advantage from the subelliptic estimate
and iterating the process we will obtain terms which give analytic growth,
of course modulo terms that give analytic growth
and of the same form of $J_{1},\dots, J_{4}$, right hand side of \eqref{eq:Est_T_S1}.

\textbf{Term $J_{4}$} on the right hand side of \eqref{eq:Est_T_S1}. 
To bound this term we will adopt the same strategies used to estimate the term $J_{5}$.
We have
\begin{linenomath}
	\begin{multline*}
	J_{4}=\sum_{\nu =1}^{q} \longsum[7]_{j=n_{1}+1 }^{n} 
	\binom{q}{\nu} 
	|\langle E_{\ell} \left( D_{t_{r}}^{\nu} b_{j}(t) \right)  D^{q-\nu}_{t_{r}} D_{t_{j}} u, E_{\ell} D^{q}_{t_{r}} u\rangle |
	\\
	\leq
	C_{1}^{2}\sum_{\nu =1}^{q} \longsum[7]_{j=n_{1}+1 }^{n} 
	\binom{q}{\nu} 
	\| \left( D_{t_{r}}^{\nu} b_{j}(t) \right)  D^{q-\nu}_{t_{r}} D_{t_{j}} u\|_{0} \| D^{q}_{t_{r}} u\|_{0}
	\\
	\leq
	C_{1}^{2} n \sum_{\nu =1}^{q-1} \longsum[7]_{j=n_{1}+1 }^{n} 
	\left(\frac{q!}{\nu!(q-\nu)!}\right)^{2} 2^{\nu}
	\| \left( D_{t_{r}}^{\nu} b_{j}(t)\right) D^{q-\nu}_{t_{r}} D_{t_{j}}u\|_{0}^{2}  
	\\
	+
	 C_{2}^{2(q+1)} q!^{2} \longsum[7]_{j=n_{1}+1 }^{n} \| D_{t_{j}}u\|_{0}^{2} 
	 +
	 \| D^{q}_{t_{r}} u\|_{0}^{2},
	\end{multline*}
\end{linenomath}
where $C_{2}$ is a suitable positive constant. 
In order to estimate the $L^{2}$-norms in the first sum
we take advantage by the assumption $\mathbf{(A3)}$, \eqref{A3_1} and \eqref{A3_2}:
\begin{linenomath}
	\begin{equation*}
	| b_{j}(t) | \leq C_{\star} a_{j,j}(t').
	\end{equation*}
\end{linenomath}
As before, we need a partition of unity in $\mathbb{T}^{n}$.
We are gong to estimate 
\begin{linenomath}
	$$
	 \|\rho_{s}(t) \left( D_{t_{r}}^{\nu} b_{j}(t)\right) D^{q-\nu}_{t_{r}} D_{t_{j}}u\|_{0}^{2} .
	 $$
\end{linenomath}
Since we use the same argument used in the case $J_{5}$, see \eqref{eq:arr}, \eqref{eq:cp_1} and \eqref{eq:cp_2},
we don't sink in details. In $\Omega_{s}$ we may write
\begin{linenomath}
	\begin{equation}
	\label{eq:bj_1}
	b_{j}(t) = \widetilde{b}_{j}(t', t'') \text{\textcursive{p}}_{j} (t_{1}, \widetilde{t}\,)
	\widetilde{\text{\textcursive{q}}}_{j,j}(t_{1}, \widetilde{t}, t''),
	\end{equation}
\end{linenomath}
where $\widetilde{b}_{j}(t)$ is an analytic function no where vanishing in $\Omega_{s}$,
$\text{\textcursive{p}}_{j} (t_{1}, \widetilde{t}\,) $ is of the same form as in \eqref{eq:arr}
and $\widetilde{\text{\textcursive{q}}}_{j,j}(t_{1}, \widetilde{t}, t'')$ 
denotes a Weierstrass polynomial of suitable degree.
Since the dependence of $b_{j}(t)$ with respect to $t''$ is confined to the factor
$\widetilde{b}_{j}(t)$ and $\widetilde{\text{\textcursive{q}}}_{j,j}(t)$, we have 
\begin{linenomath}
	\begin{equation}
	\label{eq:Dbj1}
	D_{t_{r}}^{\nu} b_{j}(t) =
	\left[ D_{t_{r}}^{\nu}  (\widetilde{b}_{j}(t', t'') \widetilde{\text{\textcursive{q}}}_{j,j}(t_{1}, \widetilde{t}, t'') ) \right] 
	\left[\widetilde{a}_{j,j}(t')\right]^{-1} a_{j,j}(t'),
	\end{equation}
\end{linenomath}
where $ \widetilde{a}_{j,j}(t') $ is an analytic function no where vanishing in $\Omega_{s}$
as $\widetilde{a}_{r,r}(t')$ in \eqref{eq:arr}.\\
We obtain
\begin{linenomath}
	\begin{multline*}
	 \longsum[7]_{j=n_{1}+1 }^{n} 
	\left(\frac{q!}{\nu!(q-\nu)!}\right)^{2} 2^{\nu}
	\| \rho_{s}(t)\left( D_{t_{r}}^{\nu} b_{j}(t)\right) D^{q-\nu}_{t_{r}} D_{t_{j}}u\|_{0}^{2}  
	\\
	=
	 \longsum[7]_{j=n_{1}+1 }^{n} 
	\left(\frac{q!}{\nu!(q-\nu)!}\right)^{2}\! 2^{\nu}
	\| \rho_{s}(t) \!\left[ D_{t_{r}}^{\nu}  (\widetilde{b}_{j}(t) \widetilde{\text{\textcursive{q}}}_{j,j}(t) ) \right] 
	\left[\widetilde{a}_{j,j}(t')\right]^{-1} a_{j,j}(t') D_{t_{j}} D^{q-\nu}_{t_{r}} u\|_{0}^{2} 
	\\
	\leq
	\longsum[7]_{j=n_{1}+1 }^{n} 2^{\nu} C^{2(\nu+1)}_{b} q^{2\nu} \|P^{j} D_{t_{r}}^{q-\nu} u\|_{0}^{2}
	\\
	+
	\longsum[7]_{j=n_{1}+1 }^{n} \sum_{p_{1}=1 }^{m}
	C^{2(\nu+1)}_{b} q^{2\nu} \|a_{n+p_{1},j} (t) D_{x_{p_{1}}}D_{t_{r}}^{q-\nu} u\|_{0}^{2}.
	\end{multline*}
\end{linenomath}
In the first term $\nu$ $t_{r}$-derivatives are turned into a factor $q^{\nu}$;
moreover it is ready to restart the process using the sub elliptic estimate.
About the second term, it can be estimate recursively as done in \eqref{eq:J_6-3}.
We have 
\begin{linenomath}
	\begin{multline*}
	J_{4}
	\leq
	\sum_{\nu =1}^{q} 
	\longsum[7]_{j=n_{1}+1 }^{n} \widetilde{C}^{2(\nu+1)}_{b} q^{2\nu} \|P^{j} D_{t_{r}}^{q-\nu} u\|_{0}^{2}
	\\
	+
	\sum_{\nu =1}^{q} 
	\sum_{i=1}^{q-1}\sum_{\nu =1}^{q-i}
	\sum_{p_{1} =1}^{m}\cdots \sum_{p_{i} =1}^{m}
	C_{\star}^{2i}\widetilde{C}_{b}^{2(\nu+1)} q^{2\nu}
	\| P^{r}  D^{q-\nu-i}_{t_{r}} D_{x_{p_{1}}}\cdots D_{x_{p_{i}}}u\|_{0}^{2}
	\\
	+
	\sum_{i =1}^{q-1}
	\sum_{p_{1} =1}^{m}\cdots \sum_{p_{i+1} =1}^{m}
	C_{\star}^{2i}\widetilde{C}_{b}^{2(q-i+1)} q^{2(q-i)}
	\| D_{x_{p_{1}}} \cdots D_{x_{p_{i+1}}} u\|_{0}^{2}
	\\
	+
	C_{2}^{2(q+1)} q!^{2} \longsum[7]_{j=n_{1}+1 }^{n} \| D_{t_{j}}u\|_{0}^{2} 
	+
	\| D^{q}_{t_{r}} u\|_{0}^{2}.
	\end{multline*}
\end{linenomath}
The third and the fourth terms on the right hand side give analytic growth rate,
and, by \eqref{eq:Est_Rm}, the last term gives analytic growth
plus a term that can be absorbed on the left side of \eqref{eq:Est_Dt}.
Concerning the terms in the first two sums we can restart the process taking advantage from subelliptic estimate;
after a suitable number of steps we will obtain terms of the form \eqref{eq:J_6-iter}, which give analytic growth. 
So, modulo terms of the same form of $J_{1},J_{2}$ and $ J_{3}$ in the right hand side of \eqref{eq:Est_T_S1},
we obtain analytic growth.

\textbf{Term $J_{1}$} on the right hand side of \eqref{eq:Est_T_S1}.
We preface that to handle this case we will use a strategy close to that one used to treat the terms $J_{5}$ and $J_{4}$.\\
We have
\begin{linenomath}
	\begin{multline*}
	J_{1}=\sum_{\nu =1}^{q} \longsum[7]_{j,p = 1}^{m}  \binom{q}{\nu} 
	|\langle E_{\ell} \left( D_{t_{r}}^{\nu} a_{n+j,n+p}(t)\right)  D^{q-\nu}_{t_{r}} D_{x_{j}} D_{x_{p}} u, E_{\ell} D^{q}_{t_{r}} u\rangle |
	\\
	\leq
	\sum_{\nu =1}^{q-1} \longsum[7]_{j,p = 1}^{m}  \binom{q}{\nu} 
	|\langle E_{\ell} \left( D_{t_{r}}^{\nu} a_{n+j,n+p}(t)\right)  D^{q-\nu}_{t_{r}} D_{x_{j}} D_{x_{p}} u, E_{\ell} D^{q}_{t_{r}} u\rangle |
	\\
	\hspace{5em}
	+
	C_{2}^{2(q+1)} q^{2q} \longsum[7]_{j,p = 1}^{m} \| D_{x_{j}} D_{x_{p}} u\|_{0}^{2} + \| D^{q}_{t_{r}} u\|_{0}^{2}
	\\
	\leq
	\sum_{\nu =1}^{q-2} \longsum[7]_{j,p = 1}^{m}  \binom{q}{\nu} 
	|\langle E_{\ell} \left( D_{t_{r}}^{\nu} a_{n+j,n+p}(t)\right)  D^{q-\nu-1}_{t_{r}} D_{x_{j}} D_{x_{p}} u, E_{\ell} D^{q+1}_{t_{r}} u\rangle |
	\\
	+
	\sum_{\nu =1}^{q-2} \longsum[7]_{j,p = 1}^{m}  \binom{q}{\nu} 
	|\langle E_{\ell} \left( D_{t_{r}}^{\nu+1} a_{n+j,n+p}(t)\right)  D^{q-\nu-1}_{t_{r}} D_{x_{j}} D_{x_{p}} u, E_{\ell} D^{q}_{t_{r}} u\rangle |
	\\
	+
	\longsum[7]_{j,p = 1}^{m}  q
	|\langle E_{\ell} \left( D_{t_{r}}^{q-1} a_{n+j,n+p}(t)\right)  D_{t_{r}} D_{x_{j}} D_{x_{p}} u, E_{\ell} D^{q}_{t_{r}} u\rangle |
	\\
	\hspace{4.5em}
	+
	C_{2}^{2(q+1)} q^{2q} \longsum[7]_{j,p = 1}^{m} \| D_{x_{j}} D_{x_{p}} u\|_{0}^{2} + \| D^{q}_{t_{r}} u\|_{0}^{2}
	.
	\end{multline*}
\end{linenomath}
We remark that, using \eqref{eq:Est_Rm}, the last two terms
on the right hand side give analytic growth
plus a term that can be absorbed on the left side of \eqref{eq:Est_Dt}.\\ 
In order to estimate the first three terms on the right hand side, as before, we use a partition of unity in $\mathbb{T}^{n}$:
let $\bigcup_{s =1}^{M} \Omega_{s}$ be a finite open chart covering of $\mathbb{T}^{n}$
and $\rho_{s}(s) \in  C_{0}^{\infty}\left( \Omega_{s}; [0,1]\right)$, $s=1,\dots, M$,
such that $\sum_{s =1}^{M} \rho_{s}(t) =1$ is a subordinate partition of unit.
We have to estimate terms of the form
\begin{linenomath}
	\begin{align}\label{eq:H123}
	&
	H_{1}= \sum_{j,p = 1}^{m} q
	|\langle E_{\ell} \rho_{s}(t) \left( D_{t_{r}}^{q-1} a_{n+j,n+p}(t)\right)  D_{t_{r}} D_{x_{j}} D_{x_{p}} u, E_{\ell} D^{q}_{t_{r}} u\rangle |;
	\\
	&\nonumber
	H_{2}=\sum_{\nu =1}^{q-2} \longsum[7]_{j,p = 1}^{m}
	\binom{q}{\nu} 
	|\langle E_{\ell} \rho_{s}(t) \left( D_{t_{r}}^{\nu+1} a_{n+j,n+p}(t)\right)  D^{q-\nu-1}_{t_{r}} D_{x_{j}} D_{x_{p}} u, E_{\ell} D^{q}_{t_{r}} u\rangle |;
	\\
	& \nonumber
	H_{3}= \sum_{\nu =1}^{q-2} \longsum[7]_{j,p = 1}^{m}
	\binom{q}{\nu} 
	|\langle E_{\ell} \rho_{s} (t) \left( D_{t_{r}}^{\nu} a_{n+j,n+p}(t)\right)  D^{q-\nu-1}_{t_{r}} D_{x_{j}} D_{x_{p}} u, E_{\ell} D^{q+1}_{t_{r}} u\rangle |.
	\end{align}
\end{linenomath}
We observe that 
by the assumption $\eqref{A3_1}$, $a_{n+k,n+k}(t) \leq C_{\star} a^{2}_{r,r}(t') $, for every $k\in \{1,\dots,m\}$ and $r\in\{n_{1},\dots, n\}$,
and since
\begin{linenomath}
	\begin{equation*}
	| a_{n+j,n+p}(t)| \leq \frac{a_{n+j,n+j}(t) +a_{n+p,n+p}(t)}{2},
	\end{equation*}
\end{linenomath}
we have
\begin{linenomath}
	\begin{equation*}
	| a_{n+j,n+p}(t)| \leq C_{\star} a^{2}_{r,r}(t').
	\end{equation*}
\end{linenomath}
Without going into further detail; using the same argument used previously, when we discussed the term $J_{4}$,
we have that
\begin{linenomath}
	\begin{equation}
	\label{eq:ajp_1}
      a_{n+j,n+p}(t) = \widetilde{a}_{n+j,n+p}(t)\,
	\widetilde{\text{\textcursive{q}}}_{j,p}(t_{1}, \widetilde{t}, t'') \text{\textcursive{p}}_{r}^{2} (t_{1}, \widetilde{t}\,),
	\end{equation}
\end{linenomath}
in $\Omega_{s}$, where $\widetilde{a}_{n+j,n+p}(t) $ is an analytic function no where vanishing in $\Omega_{s}$,
$\text{\textcursive{p}}_{r} (t_{1}, \widetilde{t}\,) $ is  as in \eqref{eq:arr}
and $\widetilde{\text{\textcursive{q}}}_{j,p}(t_{1}, \widetilde{t}, t'')$ 
denotes a Weierstrass polynomial of suitable degree.
Since the dependence of $a_{n+j,n+p}(t)$ with respect to $t''$ is confined to the first two factors,
$\widetilde{a}_{n+j,n+p}(t)$ and $\widetilde{\text{\textcursive{q}}}_{j,p}(t)$, we have 
\begin{linenomath}
	\begin{equation}
	\label{eq:Dajp_1}
	D_{t_{r}}^{\nu}  a_{n+j,n+p}(t)= 
	\left[ D_{t_{r}}^{\nu}  (     \widetilde{a}_{n+j,n+p}(t)\,
	\widetilde{\text{\textcursive{q}}}_{j,p}(t_{1}, \widetilde{t}, t'')      ) \right] 
	\left[\widetilde{a}_{r,r}(t')\right]^{-2} a_{r,r}^{2}(t'),
	\end{equation}
\end{linenomath}
where $ \widetilde{a}_{r,r}(t') $ is an analytic function no where vanishing in $\Omega_{s}$, \eqref{eq:arr}.\\
We set $\widetilde{a}_{n+j,n+p}(t)\, \widetilde{\text{\textcursive{q}}}_{j,p}(t_{1}, \widetilde{t}, t'') =\dbtilde{\text{\textcursive{q}}}_{\,j,p}(t) $.\\
Estimate of the terms $H_{1}$ and $H_{2}$ in \eqref{eq:H123}: taking advantage by \eqref{eq:Dajp_1}
and recalling that $ a_{r,r}(t') D_{t_{r}}= 2^{-1} P^{r} - \sum_{p=1}^{m} a_{n+p,r}(t) D_{x_{p}}$,
we have
\begin{linenomath}
	\begin{multline}\label{eq:Est_H12}
	H_{1} +H_{2}
	\\
	=
	\sum_{j,p = 1}^{m} q
	|\langle E_{\ell} \rho_{s}(t) \left( D_{t_{r}}^{q-1}  \dbtilde{\text{\textcursive{q}}}_{\,j,p}(t)  \right) 
	\left[\widetilde{a}_{r,r}(t')\right]^{-2} a_{r,r}^{2}(t')
	 D_{t_{r}} D_{x_{j}} D_{x_{p}} u, E_{\ell} D^{q}_{t_{r}} u\rangle |
	 \\
	+
	\sum_{\nu =1}^{q-2} \longsum[7]_{j,p = 1}^{m}
	\binom{q}{\nu} 
	|\langle E_{\ell} \rho_{s}(t) \! \left(\!  D_{t_{r}}^{\nu+1}  \dbtilde{\text{\textcursive{q}}}_{\,j,p}(t) \!\right) \!
	\left[\widetilde{a}_{r,r}(t')\right]^{-2} \! a_{r,r}^{2}(t')
	 D^{q-\nu-1}_{t_{r}} D_{x_{j}} D_{x_{p}} u,
	 \\
	  \hspace{28em}E_{\ell} D^{q}_{t_{r}} u\rangle |
	 \\
	 \leq
	 C_{1}^{2}  C_{a}^{q} q! 
	 \left( 
	 \sum_{j,p = 1}^{m}   
	 \| P^{r} D_{x_{j}} D_{x_{p}} u\|_{0}
	 +
	 \sum_{j,p_{1},p_{2} = 1}^{m}  
	 \| D_{x_{j}} D_{x_{p_{1}}} D_{x_{p_{2}}} u\|_{0} 
	 \right) \|D^{q}_{t_{r}} u\|_{0}
	 \\
	 +
	 \sum_{\nu =1}^{q-2} 
	 \binom{q}{\nu} C_{1}^{2} C_{a}^{\nu+2} 2^{\nu} (\nu+1)! 
	 \Bigg(
	 \longsum[7]_{j,p = 1}^{m}
	 \| P^{r} D_{t_{r}}^{q-\nu-2} D_{x_{j}} D_{x_{p}} u\|_{0} 
	 \\
	 \qquad \qquad
	 +
	 \longsum[7]_{j,p_{1} p_{2} = 1}^{m}
	 \| a_{n+p_{2},r} (t)D_{t_{r}}^{q-\nu-2} D_{x_{j}} D_{x_{p_{1}}} D_{x_{p_{2}}}u\|_{0} 
	 \Bigg)
	 \frac{  \|D^{q}_{t_{r}} u\|_{0}  }{2^{\nu}} 
	 \\
	 \leq 
	 \widetilde{C}_{a}^{2q} q!^{2} \left( 
	 \sum_{j,p = 1}^{m}  \| P^{r} D_{x_{j}} D_{x_{p}} u\|_{0}^{2} 
	 +  \sum_{j,p_{1},p_{2} = 1}^{m}  \| D_{x_{j}} D_{x_{p_{1}}} D_{x_{p_{2}}} u\|_{0}^{2} 
	  \right)
	  \\
	  +
	  \sum_{\nu =1}^{q-2} 
	  \widetilde{C}_{a}^{2(\nu+2)} q^{2\nu} 
	  \Bigg(
	  \longsum[7]_{j,p_{1} = 1}^{m}
	   \| P^{r} D_{t_{r}}^{q-\nu-2} D_{x_{j}} D_{x_{p_{1}}} u\|_{0}^{2}
	   \\
	   +
	   \longsum[7]_{j,p_{1} p_{2} = 1}^{m}
	   \| a_{n+p_{2},r} (t)D_{t_{r}}^{q-\nu-2} D_{x_{j}} D_{x_{p_{1}}} D_{x_{p_{2}}}u\|_{0}^{2}\Bigg)
	   + 2 \|D^{q}_{t_{r}} u\|_{0}^{2},
	\end{multline}
\end{linenomath}
where $\widetilde{C}_{a}$ is a suitable positive constant independent of $q$ and $\nu$.
We point out that to obtain the last inequality we use that $ \nu + 1 \leq
2^{\nu} $ for $ \nu \geq 0 $, $(\nu+1)! \leq 2^{\nu} \nu!$.
We remark that the first term on the right hand side gives analytic growth
and that, by \eqref{eq:Est_Rm}, the last term gives analytic growth
plus a term that can be absorbed on the left side of \eqref{eq:Est_Dt}.
Concerning the third term on the right hand side of the above inequality,
we can not apply directly the subelliptic estimate, we need to use  
assumption \eqref{A3_1}, more precisely that $|a_{n+p_{2},r} (t)| \leq C_{\star} a_{r,r}(t') $.
Iterating the procedure adopted previously (see the beginning of the estimate of the terms $H_{1}$ and $H_{2}$) 
this term can be bounded as follow
\begin{linenomath}
	\begin{multline}\label{eq:Est_H12_1}
	\sum_{\nu =1}^{q-2} \longsum[7]_{j,p_{1} p_{2} = 1}^{m}
	\widetilde{C}_{a}^{2(\nu+2)} q^{2\nu} \| a_{n+p_{2},r} (t)D_{t_{r}}^{q-\nu-2} D_{x_{j}} D_{x_{p_{1}}} D_{x_{p_{2}}}u\|_{0}^{2}
	\\
	\leq
	\widetilde{C}_{a}^{2q)} q^{2(q-2)} \| a_{n+p_{2},r} (t) D_{x_{j}} D_{x_{p_{1}}} D_{x_{p_{2}}}u\|_{0}^{2}
	\\
	+
	\sum_{\nu =1}^{q-3} \sum_{i=1}^{q-\nu-3} 
	\sum_{j=1}^{m} \sum_{p_{1} = 1}^{m}\!\!\cdots\!\! \sum_{p_{i+1} = 1}^{m} 	\widetilde{C}_{a}^{2(\nu+2)}  C_{\star}^{2i}  q^{2\nu}
     \| P^{r} D_{t_{r}}^{q-\nu-2-i} D_{x_{j}} D_{x_{p_{1}}} \cdots D_{x_{p_{i+1}}}u\|_{0}^{2}
     \\
     +
     \sum_{\nu =1}^{q-3} 
     \sum_{j=1}^{m} \sum_{p_{1} = 1}^{m}\cdots \!\!\!\! \!\!\!\sum_{p_{q-\nu-2} = 1}^{m} 	
     \!\!\! \!\!\! \! \widetilde{C}_{a}^{2(\nu+2)}  
     \!C_{\star}^{2(q-\nu-2)}  q^{2\nu}
     \| a_{n+p_{q-\nu-2},r} (t)  D_{x_{j}} D_{x_{p_{1}}} \!\!\cdots D_{x_{p_{q-\nu-2}} } u\|_{0}^{2}.
	\end{multline}
\end{linenomath}
The first and the last terms on the right hand side give analytic growth,
indeed by \eqref{eq:Dx_Est} we have that $ \|  D_{x_{j}} D_{x_{p_{1}}} \cdots D_{x_{p_{q-\nu-2}} } u\|_{0}^{2} \leq C_{2}^{2(q-\nu)} q^{2(q-\nu-1)}$,
where $C_{2}$ is a positive constant. Concerning the terms in the second sum
we remark that $\nu$ $t_{r}$-derivatives were turned into a factor $q^{\nu}$
and $i+2$ $t_{r}$-derivatives were turned into $i+2$ $x$-derivatives.
About their estimate, we can restart the process applying the subelliptic
estimate,  after $s$-times we will face of with terms of the form \eqref{eq:J_6-iter_0};
so, iterating a suitable number of times the process we will  obtain terms of the form \eqref{eq:J_6-iter},
which give analytic growth, modulo terms of the form $H_{3}$, $J_{2}$ and $J_{3}$,
which we will see shortly how to treat.\\ 
Term $H_{3}$ in \eqref{eq:H123}: we will analyze the case $E_{\ell}$, with $\ell\in \{ 1,\dots,N\}$,
i.e. when $E_{\ell} \neq id$. The case $E_{0}=id$ can be treated in the same way without
the extra terms that arise from the commutators between $E_{\ell}$,  and $a_{r,r}(t')$.
By \eqref{eq:Dajp_1}, we have 
\begin{linenomath}
	\begin{multline}\label{eq:Est_H3}
	H_{3}
	\\
	= \sum_{\nu =1}^{q-2} \longsum[7]_{j,p = 1}^{m}
	\binom{q}{\nu} 
	|\langle E_{\ell} \rho_{s} (t)\! 
	\left( D_{t_{r}}^{\nu}  \dbtilde{\text{\textcursive{q}}}_{\,j,p}(t)\right) \!
	\left[\widetilde{a}_{r,r}(t')\right]^{-2} \!a_{r,r}^{2}(t')
	 D^{q-\nu-1}_{t_{r}} D_{x_{j}} D_{x_{p}} u,
	 \\ 
	 \hspace{29em} E_{\ell} D^{q+1}_{t_{r}} u\rangle |
	 \\
	 \leq
	 \sum_{\nu =1}^{q-2} \longsum[7]_{j,p = 1}^{m}
	 \binom{q}{\nu} 
	 |\langle E_{\ell} \rho_{s} (t) \!
	 \left( \!D_{t_{r}}^{\nu}  \dbtilde{\text{\textcursive{q}}}_{\,j,p}(t) \right) \!
	 \left[\widetilde{a}_{r,r}(t')\right]^{-2} \! a_{r,r}(t')
	 D^{q-\nu-1}_{t_{r}} D_{x_{j}} D_{x_{p}} u, 
	 \\
	 \hspace{26em}
	 E_{\ell} a_{r,r}(t') D^{q+1}_{t_{r}} u\rangle |
	 \\
	 +
	 \sum_{\nu =1}^{q-2} \longsum[7]_{j,p = 1}^{m}
	 \binom{q}{\nu} 
	 |\langle \rho_{s} (t) 
	 \left(\! D_{t_{r}}^{\nu}  \dbtilde{\text{\textcursive{q}}}_{\,j,p}(t) \right) \! 
	 \left[\widetilde{a}_{r,r}(t')\right]^{-2} \! a_{r,r}(t')
	 D^{q-\nu-1}_{t_{r}} D_{x_{j}} D_{x_{p}} u, 
	 \\
	 \hspace{23em}
	 [ a_{r,r}(t'),E_{\ell}^{*}E_{\ell} ] D^{q+1}_{t_{r}} u\rangle |
	 \\
	 =
	 H_{3,1} + H_{3,2}, 
	\end{multline}
\end{linenomath}
where  $E_{\ell}^{*} $ is the adjoint of $E_{\ell}$.\\
We begin to analyze the last term, $H_{3,2}$.
Since $[ a_{r,r}(t'),E_{\ell}^{*}E_{\ell} ] $
is a pseudodifferential operator of order $-1$ we have that $[ a_{r,r}(t'),E_{\ell}^{*}E_{\ell} ] D_{t_{r}}$
is a zero order pseudodifferential operator.
Let $C_{3}>0$ such that $\|[ a_{r,r}(t'),E_{\ell}^{*}E_{\ell} ] D_{t_{r}}\|_{L^{2} \rightarrow L^{2}} \leq C_{3}$,
we have 
\begin{linenomath}
	\begin{multline}\label{eq:Est_H323}
	 H_{3,2} 
	\leq
	\sum_{\nu =1}^{q-2} \longsum[7]_{j,p = 1}^{m}
	m^{2} C_{1}C_{3} 2^{\nu+1} C_{a}^{2(\nu+1)} q^{2\nu}
	\|P^{r} D^{q-\nu-2}_{t_{r}} D_{x_{j}} D_{x_{p}} u \|_{0}^{2}
	\\
	+
	\sum_{\nu =1}^{q-2} \sum_{j,p_{1},p_{2} = 1}^{m}
	m^{2} C_{1}C_{3} 2^{\nu+1} C_{a}^{2(\nu+1)} q^{2\nu}
	\| a_{n+p_{2},r}(t')
	D^{q-\nu-2}_{t_{r}} D_{x_{j}} D_{x_{p_{1}}} D_{x_{p_{1}}} u \|_{0}^{2}
	\\
	+
	2\| D^{q}_{t_{r}} u\|_{0}^{2}.
	\end{multline}
\end{linenomath}
We remark that the first two terms on the right hand side have the same form of terms
already met in the estimate of $H_{1}$ and $H_{2}$, second and third term
on the right hand side of \eqref{eq:Est_H12}.
In particular the second term can be estimated as done in \eqref{eq:Est_H12_1}.
By \eqref{eq:Est_Rm}, the last term gives analytic growth
plus a term that can be absorbed on the left side of \eqref{eq:Est_Dt}.
So, $H_{3,2}$ is estimated by terms which give analytic growth,
modulo terms of the form $H_{3,1}$, $J_{2}$ and $J_{3}$ and a term that can be absorbed on the lest hand side of \eqref{eq:Est_Dt}.\\
Term $H_{3,1}$, we have
\begin{linenomath}
	\begin{multline}\label{eq:Est_H31}
	H_{3,1} \leq 
	\sum_{\nu =1}^{q-2} \longsum[7]_{j,p_{1} = 1}^{m}
	C_{1}^{2} C_{a}^{\nu+1} q^{\nu} 
	\Bigg(
	\| P^{r} D^{q-\nu-2}_{t_{r}} D_{x_{j}} D_{x_{p_{1}}} u\|_{0} \| P^{r} D^{q}_{t_{r}} u\|_{0}
	\\
	+ \sum_{q_{1}=1}^{m}  \| P^{r} D^{q-\nu-2}_{t_{r}} D_{x_{j}} D_{x_{p_{1}}} u\|_{0} \| a_{n+q_{1},r} (t)D^{q}_{t_{r}} D_{x_{q_{1}}} u\|_{0}
	\\
	+
	 \sum_{p_{2}=1}^{m}  \| a_{n+p_{2},r} (t)D^{q-\nu-2}_{t_{r}} D_{x_{j}} D_{x_{p_{1}}}  D_{x_{p_{2}}} u\|_{0} \|  P^{r} D^{q}_{t_{r}} D_{x_{q_{1}}} u\|_{0}
	\\
	+
	\sum_{p_{2}=1}^{m} \sum_{q_{1}=1}^{m}  \| a_{n+p_{2},r} (t)D^{q-\nu-2}_{t_{r}} D_{x_{j}} D_{x_{p_{1}}}  D_{x_{p_{2}}} u\|_{0}
	\| a_{n+q_{1},r} (t)D^{q}_{t_{r}} D_{x_{q_{1}}} u\|_{0}
	\Bigg)
	\\
	\leq
	2 \sum_{\nu =1}^{q-2} \longsum[7]_{j,p_{1} = 1}^{m}
	C_{1}^{2} m^{2} \left( C_{\nu} +1 \right) C_{a}^{\nu+1} q^{\nu} \| P^{r} D^{q-\nu-2}_{t_{r}} D_{x_{j}} D_{x_{p_{1}}} u\|_{0}^{2}
	\\
	+2
	\sum_{\nu =1}^{q-2} \longsum[7]_{j,p_{1} = 1}^{m} \sum_{p_{2}=1}^{m} 
	C_{1}^{2} m^{3} \left( C_{\nu} +1 \right) C_{a}^{\nu+1} q^{\nu} 
	\| a_{n+p_{2},r} (t)D^{q-\nu-2}_{t_{r}} D_{x_{j}} D_{x_{p_{1}}}  D_{x_{p_{2}}} u\|_{0}^{2}
	\\
	+
	2 \sum_{\nu =1}^{q-2} \sum_{q_{1}=1}^{m}
	m^{2} \| a_{n+q_{1},r} (t)D^{q}_{t_{r}} D_{x_{q_{1}}} u\|_{0}^{2}
	+
	2 \sum_{\nu =1}^{q-2} C_{\nu}^{-1} \| P^{r} D^{q}_{t_{r}} u\|_{0}^{2},
	\end{multline}
\end{linenomath}
where $C_{\nu}$, $\nu=1,\dots,q-2$, are arbitrary constants. We choose
$ C_{\nu} = \delta^{-1}_{1} 2^{\nu+2}  $ where $\delta_{1}$ is a fixed small constant
such that $(N+1)C \delta_{1} \leq 1$, here $C$ is the constant in the right hand side of the subelliptic
estimate \eqref{eq:Est_Dt}. 
Since $\sum_{\nu=1}^{q-2} C_{\nu}^{-1} \leq \delta_{1} 2^{-2}$, we
conclude that the last term on the right hand side of the above estimate can be absorbed on the left hand side of \eqref{eq:Est_Dt}.
Concerning the other terms, they have a form already met previously,
more precisely the second and the third term can be estimate as done in \eqref{eq:Est_H12_1},
obtaining terms that give analytic growth rate directly or terms,
as the first on the right hand side, that can be estimated restarting the process via the subelliptic estimate. 
About these terms, as showed previously, after a suitable number of iterations
they will give terms of the form \eqref{eq:J_6-iter}, 
which give analytic growth, modulo terms of the form as $J_{2}$ and $J_{3}$.\\
Summing up, the term $J_{1}$ can be estimate by terms that give analytic growth
modulo terms that can be absorbed on the right hand side of \eqref{eq:Est_Dt}
or terms of the form as $J_{2}$ and $J_{3}$, in the right hand side of \eqref{eq:Est_T_S1}.

\textbf{Term $J_{3}$} on the right hand side of \eqref{eq:Est_T_S1}.
We proceed in a similar way to that used to treat $J_{1}$.
We have
\begin{linenomath}
	\begin{multline}\label{eq:Est_J3-1}
	J_{3}=
	2 \sum_{\nu =1}^{q}
	\longsum[7]_{j =n_{1}+1}^{n} \sum_{p =1}^{m} \binom{q}{\nu} 
	|\langle E_{\ell} \left( D_{t_{r}}^{\nu} a_{j,n+p} (t) \right)  D^{q-\nu}_{t_{r}} D_{t_{j}} D_{x_{p}} u,  E_{\ell} D^{q}_{t_{r}} u\rangle |
	\\
	\leq
	2\sum_{\nu =1}^{q-1}
	\longsum[7]_{j =n_{1}+1}^{n} \sum_{p =1}^{m} \binom{q}{\nu} 
	|\langle E_{\ell} \left( D_{t_{r}}^{\nu} a_{j,n+p} (t) \right)  D^{q-\nu}_{t_{r}}  D_{x_{p}} u,  E_{\ell} D_{t_{j}} D^{q}_{t_{r}} u\rangle |
	\\
	+
	2\sum_{\nu =1}^{q-1}
	\longsum[7]_{j =n_{1}+1}^{n} \sum_{p =1}^{m} \binom{q}{\nu} 
	|\langle E_{\ell} \left( D_{t_{j}} D_{t_{r}}^{\nu} a_{j,n+p} (t) \right)  D^{q-\nu}_{t_{r}}  D_{x_{p}} u,  E_{\ell} D^{q}_{t_{r}} u\rangle |
	\\
	+
	\longsum[7]_{j =n_{1}+1}^{n} \sum_{p =1}^{m}
	|\langle E_{\ell} \left( D_{t_{r}}^{q} a_{j,n+p} (t) \right)  D_{t_{j}} D_{x_{p}} u,  E_{\ell} D^{q}_{t_{r}} u\rangle |.
	\end{multline}
\end{linenomath}
The last term gives analytic growth modulo a term that can be absorbed on the left hand side of the
subelliptic estimate \eqref{eq:Est_Dt}.
In order to estimate the other two terms, we use a partition of unity in $\mathbb{T}^{n}$, as already done several times,
and the assumption $\mathbf{(A3)}$, \eqref{A3_1} and \eqref{A3_2}:
\begin{linenomath}
	\begin{equation*}
	|a_{j,n+p}(t)| \leq C_{\star} \left(a_{k,k}(t')\right)^{2} \text{ and } \longsum[7]_{i=n_{1}+1}^{n} a_{i,i}(t') \leq C_{\star} a_{k,k}(t'),
	\end{equation*}
\end{linenomath}
for every $j,k \in \{n_{1},\dots,n  \}$ and $p\in \{1, \dots, m \}$.\\ 
Let $\Omega_{s}$ and $\rho_{s}(t)$ be as before. Using the same strategy adopted
to obtain \eqref{eq:cp_1}, \eqref{eq:cp_2}, \eqref{eq:ajp_1} and \eqref{eq:Dajp_1}
we have
\begin{linenomath}
	\begin{equation}
	\label{eq:ajnp_1}
	a_{j,n+p}(t) = \widetilde{a}_{j,n+p}(t)\,
	\widetilde{\text{\textcursive{q}}}_{\,j,p}(t_{1}, \widetilde{t}, t'') \text{\textcursive{p}}_{r}^{2} (t_{1}, \widetilde{t}\,),
	\end{equation}
\end{linenomath}
in $\Omega_{s}$, where $\widetilde{a}_{j,n+p}(t) $ is an analytic function no where vanishing in $\Omega_{s}$,
$\text{\textcursive{p}}_{r} (t_{1}, \widetilde{t}\,) $ is  as in \eqref{eq:arr}
and $\widetilde{\text{\textcursive{q}}}_{\, j,p}(t_{1}, \widetilde{t}, t'')$ 
denotes a Weierstrass polynomial of suitable degree.
Since the dependence of $a_{j,n+p}(t)$ with respect to $t''$ is confined to the first two factors,
$\widetilde{a}_{j,n+p}(t)$ and $\widetilde{\text{\textcursive{q}}}_{\,j,p}(t)$, we have 
\begin{linenomath}
	\begin{equation}
	\label{eq:Dajnp_1}
	D_{t_{r_{1}}}^{\nu_{1}} D_{t_{r_{2}}}^{\nu_{2}}  a_{j,n+p}(t)= 
	\left[ D_{t_{r_{1}}}^{\nu_{1}} D_{t_{r_{2}}}^{\nu_{2}}  
	( \widetilde{a}_{n+j,n+p}(t)\, \widetilde{\text{\textcursive{q}}}_{\, j,p}(t) ) \right] 
	\left[\widetilde{a}_{r,r}(t')\right]^{-2} a_{r,r}^{2}(t'),
	\end{equation}
\end{linenomath}
where $ \widetilde{a}_{r,r}(t') $, introduced in  \eqref{eq:arr}, is an analytic function no where vanishing in $\Omega_{s}$,
$r_{1},r_{2} \in \{ n_{1}+1, \dots, n\}$.\\
We set $\widetilde{a}_{j,n+p}(t)\, \widetilde{\text{\textcursive{q}}}_{\, j,p}(t_{1}, \widetilde{t}, t'') =\dbtilde{\text{\textcursive{q}}}_{\, j,p}(t) $.\\
We begin to analyze the second term on the right hand side of \eqref{eq:Est_J3-1}.
By \eqref{eq:Dajnp_1}, the assumption \eqref{A3_2}, recalling that 
$P^{k} = 2a_{k,k}(t') D_{t_{k}} + 2 \sum_{p_{1}=1}^{m} a_{k,n+p_{1}}(t) D_{x_{p_{1}}}$,  $ k=n_{1}+1,\dots, n$,
and using that $\nu+1 \leq 2^{\nu}$, we obtain
\begin{linenomath}
	\begin{multline}\label{eq:Est_J32_33}
	2\sum_{\nu =1}^{q-1}
	\longsum[7]_{j =n_{1}+1}^{n} \sum_{p =1}^{m} \binom{q}{\nu} 
	|\langle E_{\ell} \rho_{s}(t)\left( D_{t_{j}} D_{t_{r}}^{\nu} a_{j,n+p} (t) \right)  D^{q-\nu}_{t_{r}}  D_{x_{p}} u,  E_{\ell} D^{q}_{t_{r}} u\rangle |
	\\
	= 
	2\sum_{\nu =1}^{q-1}
	\longsum[7]_{j =n_{1}+1}^{n} \sum_{p =1}^{m} \binom{q}{\nu} 
	|\langle E_{\ell} \rho_{s}(t)\left( D_{t_{j}} D_{t_{r}}^{\nu} \dbtilde{\text{\textcursive{q}}}_{\, j,p}(t) \right) 
	\left[\widetilde{a}_{r,r}(t')\right]^{-2} a_{r,r}^{2}(t') D^{q-\nu}_{t_{r}}  D_{x_{p}} u,  
	\\
	\hspace{29em}
	E_{\ell} D^{q}_{t_{r}} u\rangle |
	\\
	\leq
	\sum_{\nu =1}^{q-1}
	\longsum[7]_{j =n_{1}+1}^{n} \sum_{p =1}^{m} C_{1}^{2} 2^{2\nu+1} C_{a}^{\nu+2} q^{\nu}
	\| a_{r,r}(t') D^{q-\nu}_{t_{r}}  D_{x_{p}} u\|_{0}  2^{-\nu} \| D^{q}_{t_{r}} u\|_{0}
	\\
	\leq
	\sum_{\nu =1}^{q-1}
	\longsum[7]_{j =n_{1}+1}^{n} \sum_{p =1}^{m} 
	m n_{1} C_{1}^{4} 2^{2(2\nu+1)} C_{a}^{2(\nu+2)} q^{2\nu}
	\| a_{r,r}(t') D^{q-\nu}_{t_{r}}  D_{x_{p}} u\|_{0}^{2} 
	+
	 \| D^{q}_{t_{r}} u\|_{0}^{2}
	 \\
	 \hspace{-10em}
	 \leq
	 \sum_{\nu =1}^{q-1}
	 \sum_{p =1}^{m} \widetilde{C}_{a}^{2(\nu+1)} q^{2\nu}
	 \| P^{r} D^{q-\nu-1}_{t_{r}}  D_{x_{p}} u\|_{0}^{2} 
	 \\
	 +
	 \sum_{\nu =1}^{q-1} \sum_{p_{1} =1}^{m}  \sum_{p_{2} =1}^{m} 
	 \widetilde{C}_{a}^{2(\nu+1)} q^{2\nu}
	 \| a_{r,n+p_{2}}(t') D^{q-\nu-1}_{t_{r}}  D_{x_{p_{2}}} D_{x_{p_{1}}} u\|_{0}^{2} 
	 +
	 \| D^{q}_{t_{r}} u\|_{0}^{2}
	.
	\end{multline}
\end{linenomath}
The last term can be estimate as in \eqref{eq:Est_Rm} obtaining a
term that gives analytic growth and a term which can be absorbed
on the left hand side of \eqref{eq:Est_Dt}.
The second term as the same form of the third term on the right
hand side of \eqref{eq:Est_H12}; it can be estimate as done in \eqref{eq:Est_H12_1},
producing terms that gives analytic growth modulo terms
that can be handled restarting the process, i.e. applying the sub-elliptic estimate,
and terms of the form $J_{2}$. \\
%
Estimate of the first term on the right hand side of \eqref{eq:Est_J3-1}.
As before we use the partition of unity in $\mathbb{T}^{n}$: 
\begin{linenomath}
	\begin{multline}\label{eq:Est_J31}
	2\sum_{\nu =1}^{q-1}
	\longsum[7]_{j =n_{1}+1}^{n} \sum_{p =1}^{m} \binom{q}{\nu} 
	|\langle E_{\ell} \rho_{s}(t)\left( D_{t_{r}}^{\nu} a_{j,n+p} (t) \right)  D^{q-\nu}_{t_{r}}  D_{x_{p}} u,  E_{\ell} D_{t_{j}} D^{q}_{t_{r}} u\rangle |
	\\
	=
	2\sum_{\nu =1}^{q-1}
	\longsum[7]_{j =n_{1}+1}^{n} \sum_{p =1}^{m} \binom{q}{\nu} 
	|\langle E_{\ell} \rho_{s}(t)
	\left(  D_{t_{r}}^{\nu} \dbtilde{\text{\textcursive{q}}}_{\, j,p}(t) \right) 
	\left[\widetilde{a}_{r,r}(t')\right]^{-2} a_{r,r}^{2}(t') 
	 D^{q-\nu}_{t_{r}}  D_{x_{p}} u,  
	 \\
	 \hspace{27em}
	 E_{\ell} D_{t_{j}} D^{q}_{t_{r}} u\rangle |
	 \\
	 \leq
	 2\sum_{\nu =1}^{q-1}
	 \longsum[7]_{j =n_{1}+1}^{n} \sum_{p =1}^{m} \binom{q}{\nu} 
	 |\langle E_{\ell} \rho_{s}(t)
	 \left(  D_{t_{r}}^{\nu} \dbtilde{\text{\textcursive{q}}}_{\, j,p}(t) \right) 
	 \left[\widetilde{a}_{r,r}(t')\right]^{-2} a_{r,r}(t') 
	 D^{q-\nu}_{t_{r}}  D_{x_{p}} u, 
	  \\
	 \hspace{25em}
	  E_{\ell} a_{r,r}(t') D_{t_{j}} D^{q}_{t_{r}} u\rangle |
	 \\
	 +
	 2\sum_{\nu =1}^{q-1}
	 \longsum[7]_{j =n_{1}+1}^{n} \sum_{p =1}^{m} \binom{q}{\nu} 
	 |\langle \rho_{s}(t)
	 \left(  D_{t_{r}}^{\nu} \dbtilde{\text{\textcursive{q}}}_{\, j,p}(t) \right) 
	 \left[\widetilde{a}_{r,r}(t')\right]^{-2} a_{r,r}(t') 
	 D^{q-\nu}_{t_{r}}  D_{x_{p}} u, 
	  \\
	 \hspace{25em}
	 [ a_{r,r}(t'),E_{\ell}^{*}E_{\ell} ]D_{t_{j}} D^{q}_{t_{r}} u\rangle |.
	\end{multline}
\end{linenomath}
where  $E_{\ell}^{*} $ is the adjoint of $E_{\ell}$. We remark that
since $[ a_{r,r}(t'),E_{\ell}^{*}E_{\ell} ] $ is a pseudodifferential operator of order $-1$ we have that $[ a_{r,r}(t'),E_{\ell}^{*}E_{\ell} ] D_{t_{r}}$
is a zero order pseudodifferential operator and moreover there is a positive constant $C_{3}$,
such that $\|[ a_{r,r}(t'),E_{\ell}^{*}E_{\ell} ] D_{t_{r}}\|_{L^{2} \rightarrow L^{2}} \leq C_{3}$.
By the assumption \eqref{A3_2}, the identity 
$P^{k} = 2a_{k,k}(t') D_{t_{k}} + 2 \sum_{p_{1}=1}^{m} a_{k,n+p_{1}}(t) D_{x_{p_{1}}}$,  $ k=n_{1}+1,\dots, n$,
and using the same strategy adopted to estimate $H_{3}$, \eqref{eq:Est_H323} and \eqref{eq:Est_H31},
the right hand side of above inequality can be estimate  by
\begin{linenomath}
	\begin{multline}
	\sum_{\nu =1}^{q-1} \sum_{p_{1} =1}^{m}
	C_{4,\delta_{1}} C_{a}^{2(\nu+1)} q^{2\nu} \| P^{r} D^{q-\nu-1 }_{t_{r}}  D_{x_{p_{1}}} u\|_{0}^{2}
	\\
	+
	\sum_{\nu =1}^{q-1} \sum_{p_{1} =1}^{m} \sum_{p_{2} =1}^{m}
	C_{4,\delta_{1}} C_{a}^{2(\nu+1)} q^{2\nu} \| a_{n+p_{2} } (t) D^{q-\nu-1 }_{t_{r}}  D_{x_{p_{1}}} D_{x_{p_{2}}}u\|_{0}^{2}
	\\
	+
	\longsum[7]_{j =n_{1}+1}^{n} \sum_{p_{2} =1}^{m} nm \| a_{n+p_{2},j} (t) D^{q }_{t_{r}} D_{x_{p_{2}}}u\|_{0}^{2}
	+
	2 \| D^{q }_{t_{r}} u\|_{0}^{2} + \delta_{1} \longsum[7]_{j =n_{1}+1}^{n} \| P^{j} D^{q }_{t_{r}}  u\|_{0}^{2},
	\end{multline}
\end{linenomath}
where $\delta_{1}$ is a suitable small constant as before and $C_{4,\delta_{1}}$ is a positive constant depending
of $\delta_{1}^{-1}$ but independent of $q$ and $\nu$. 
The last term on the right hand side can be absorbed on the left hand side of \eqref{eq:Est_Dt}.
The second to last term can be estimate as in \eqref{eq:Est_Rm} obtaining a
term that gives analytic growth and a term which can be absorbed
on the left hand side of \eqref{eq:Est_Dt}.
Concerning the other terms, they have a form already met previously,
more precisely the second and the third term can be estimate as done in \eqref{eq:Est_H12_1},
producing terms that give analytic growth rate directly or terms,
as the first, that can be estimated restarting the process via the subelliptic estimate. 
About these terms, as showed previously, after a suitable number of iterations,
they will give terms of the form \eqref{eq:J_6-iter}, that gives analytic growth.

Summing up, the term $J_{3}$ can be estimate by terms that give analytic growth,
modulo terms that can be absorbed on the right hand side of  \eqref{eq:Est_Dt}
or of the same form of $J_{2}$ in the right hand side of  \eqref{eq:Est_T_S1}.

\textbf{Term $J_{2}$} on the right hand side of \eqref{eq:Est_T_S1}.
We proceed in a similar way to that used to treat $J_{1}$ and $J_{3}$.
We have
\begin{linenomath}
	\begin{multline*}
	J_{2}=2 \sum_{\nu =1}^{q}
	\sum_{j =1}^{n_{1}} \sum_{p =1}^{m} \binom{q}{\nu} 
	|\langle E_{\ell} \left( D_{t_{r}}^{\nu} a_{j,n+p} (t) \right)  D^{q-\nu}_{t_{r}} D_{t_{j}} D_{x_{p}} u, E_{\ell}D^{q}_{t_{r}} u\rangle |
	\\
	\leq
	2 \sum_{\nu =1}^{q-1}
	\sum_{j =1}^{n_{1}} \sum_{p =1}^{m} \binom{q}{\nu} 
	|\langle E_{\ell} \left( D_{t_{r}}^{\nu} a_{j,n+p} (t) \right)  D^{q-\nu-1}_{t_{r}} D_{t_{j}} D_{x_{p}} u, E_{\ell} D_{t_{r}} D^{q}_{t_{r}} u\rangle |
	\\
	+
	2 \sum_{\nu =1}^{q-1}
	\sum_{j =1}^{n_{1}} \sum_{p =1}^{m} \binom{q}{\nu} 
	|\langle E_{\ell} \left( D_{t_{r}}^{\nu+1} a_{j,n+p} (t) \right)  D^{q-\nu-1}_{t_{r}}  D_{t_{j}} D_{x_{p}} u, E_{\ell}  D^{q}_{t_{r}} u\rangle |
	\\
	+
	\sum_{j =1}^{n_{1}} \sum_{p =1}^{m} 
	|\langle E_{\ell} \left( D_{t_{r}}^{q} a_{j,n+p} (t) \right)  D_{t_{j}} D_{x_{p}} u, E_{\ell}D^{q}_{t_{r}} u\rangle |
	=
	J_{2,1} +J_{2,2} +J_{2,3}.
	\end{multline*}
\end{linenomath}
We begin to estimate $J_{2,3}$. By \eqref{Dtj_1n}, we have 
\begin{linenomath}
	\begin{multline*}
	J_{2,3} \leq C_{2}^{2(q+1)} q^{2q} \left( \sum_{j_{1} =1}^{n_{1}} \sum_{p_{1} =1}^{m}  \| P^{j_{1}} D_{x_{p_{1}}} u\|^{2}_{0}
	+ \sum_{p_{1} =1}^{m} \sum_{p_{2} =1}^{m} \| D_{x_{p_{1}}} D_{x_{p_{2}}} u\|_{0}^{2}\right)
	\\
	+2\| D^{q}_{t_{r}} u \|_{0}^{2}.
	\end{multline*}
\end{linenomath}
The first term on the right hand side  gives analytic growth rate; the second term 
can be estimate via \eqref{eq:Est_Rm} giving a contribution of analytic growth
plus a term that can be absorbed on the left side of the \eqref{eq:Est_Dt}.\\
Terms $J_{2,2}$. By \eqref{Dtj_1n}, we have 
\begin{linenomath}
	\begin{multline*}
	J_{2,2}
	\leq
     \sum_{\nu =1}^{q-1}
	\sum_{j_{2} =1}^{n_{1}} \sum_{p_{1} =1}^{m} C_{3} (2C_{a})^{2(\nu +2)} q^{2\nu}
	\|P^{j_{2}}  D^{q-\nu-1}_{t_{r}}  D_{x_{p_{1}}} u \|_{0}^{2}
	\\
	+
	\sum_{\nu =1}^{q-1}
	\sum_{j_{2} =1}^{n_{1}} \sum_{p_{1} =1}^{m} \sum_{p_{2} =1}^{m} 
	C_{3} (2C_{a})^{2(\nu +2)} q^{2\nu}
	\| a_{n+p_{2},j_{2}} (t) D^{q-\nu-1}_{t_{r}}  D_{x_{p}} u \|_{0}^{2}
	+
    2\|  D^{q}_{t_{r}} u\|_{0}^{2},
	\end{multline*}
\end{linenomath}
where $C_{3}$ and $C_{a}$ are suitable positive constant independent of $q$ and $\nu$.
So, $J_{2,2}$ is estimated by term that we have already met several times.  
We remark that the second term can be estimate as done in \eqref{eq:Est_H12_1},
producing terms that give analytic growth rate directly or terms,
as the first, that can be estimated restarting the process via the subelliptic estimate. 
About these terms, as showed previously, after a suitable number of iterations,
they will give terms of the form \eqref{eq:J_6-iter}, that gives analytic growth.\\

Term $J_{2,1} $. We use a partition of unity in $\mathbb{T}^{n}$
and  the assumptions  \eqref{A3_1}: 
\begin{linenomath}
	\begin{equation*}
	|a_{j,n+p}(t)| \leq C_{\star} a_{r,r}(t'),
	\end{equation*}
\end{linenomath}
for every $j\in \{1,\dots,n_{1} \}$, $p\in \{1, \dots, m \}$ and $r\in \{n_{1}+1, \dots, n \}$.\\
Let $\Omega_{s}$ and $\rho_{s}(t)$ be as before. Using the same strategy adopted
to obtain \eqref{eq:cp_1}, \eqref{eq:cp_2}, \eqref{eq:ajp_1} and \eqref{eq:Dajp_1}
we have
\begin{linenomath}
	\begin{equation}
	\label{eq:ajnp_2}
	a_{j,n+p}(t) = \widetilde{a}_{j,n+p}(t)\,
	\widetilde{\text{\textcursive{q}}}_{\,j,p}(t_{1}, \widetilde{t}, t'') \text{\textcursive{p}}_{r}  (t_{1}, \widetilde{t}\,),
	\end{equation}
\end{linenomath}
in $\Omega_{s}$, where $\widetilde{a}_{j,n+p}(t) $ is an analytic function no where vanishing in $\Omega_{s}$,
$\text{\textcursive{p}}_{r} (t_{1}, \widetilde{t}\,) $ is  as in \eqref{eq:arr}
and $\widetilde{\text{\textcursive{q}}}_{\, j,p}(t_{1}, \widetilde{t}, t'')$ 
denotes a Weierstrass polynomial of suitable degree.
Since the dependence of $a_{j,n+p}(t)$ with respect to $t''$ is confined to the first two factors,
$\widetilde{a}_{j,n+p}(t)$ and $\widetilde{\text{\textcursive{q}}}_{\,j,p}(t)$, we have 
\begin{linenomath}
	\begin{equation}
	\label{eq:Dajnp_2}
	 D_{t_{r}}^{\nu}  a_{j,n+p}(t)= 
	\left[  D_{t_{r}}^{\nu}  
	( \widetilde{a}_{n+j,n+p}(t)\, \widetilde{\text{\textcursive{q}}}_{\, j,p}(t) ) \right] 
	\left[\widetilde{a}_{r,r}(t')\right]^{-1} a_{r,r}(t'),
	\end{equation}
\end{linenomath}
where $ \widetilde{a}_{r,r}(t') $ is an analytic function no where vanishing in $\Omega_{s}$, \eqref{eq:arr}.\\
We set $\widetilde{a}_{j,n+p}(t)\, \widetilde{\text{\textcursive{q}}}_{\, j,p}(t_{1}, \widetilde{t}, t'') =\dbtilde{\text{\textcursive{q}}}_{\, j,p}(t) $.\\
Taking advantage from \eqref{Dtj_1n}, we have
\begin{linenomath}
	\begin{multline*}
	2 \sum_{\nu =1}^{q-1}
	\sum_{j =1}^{n_{1}} \sum_{p =1}^{m} \binom{q}{\nu} 
	|\langle E_{\ell} \rho_{s} (t)\left( D_{t_{r}}^{\nu} a_{j,n+p} (t) \right)  D^{q-\nu-1}_{t_{r}} D_{t_{j}} D_{x_{p}} u, E_{\ell} D_{t_{r}} D^{q}_{t_{r}} u\rangle |
	\\
	=
	2 \sum_{\nu =1}^{q-1}
	\sum_{j =1}^{n_{1}} \sum_{p =1}^{m} \binom{q}{\nu} 
	|\langle E_{\ell} \rho_{s} (t)\left( D_{t_{r}}^{\nu}   \dbtilde{\text{\textcursive{q}}}_{\, j,p}(t)  \right)  
	\left[\widetilde{a}_{r,r}(t')\right]^{-1} a_{r,r}(t')
	D^{q-\nu-1}_{t_{r}} D_{t_{j}} D_{x_{p}} u,
	 \\
	\hspace{28em}
	 E_{\ell} D_{t_{r}} D^{q}_{t_{r}} u\rangle |
	\\
	\leq
	2 \sum_{\nu =1}^{q-1}
	\sum_{j =1}^{n_{1}} \sum_{p =1}^{m} \binom{q}{\nu} 
	|\langle E_{\ell} \rho_{s} (t)\left( D_{t_{r}}^{\nu}   \dbtilde{\text{\textcursive{q}}}_{\, j,p}(t)  \right)  
	\left[\widetilde{a}_{r,r}(t')\right]^{-1} 
	D^{q-\nu-1}_{t_{r}} D_{t_{j}} D_{x_{p}} u, 
	 \\
	\hspace{26em}
	E_{\ell}  a_{r,r}(t') D_{t_{r}} D^{q}_{t_{r}} u\rangle |
	\\
	+
	2 \sum_{\nu =1}^{q-1}
	\sum_{j =1}^{n_{1}} \sum_{p =1}^{m} \binom{q}{\nu} 
	|\langle \rho_{s} (t)\left( D_{t_{r}}^{\nu}   \dbtilde{\text{\textcursive{q}}}_{\, j,p}(t)  \right)  
	\left[\widetilde{a}_{r,r}(t')\right]^{-1} 
	D^{q-\nu-1}_{t_{r}} D_{t_{j}} D_{x_{p}} u, 
	 \\
	\hspace{22em}
	 [ a_{r,r}(t'), E_{\ell}^{*}E_{\ell}] D_{t_{r}} D^{q}_{t_{r}} u\rangle |.
	\end{multline*}
\end{linenomath}
%
We point out that $[ a_{r,r}(t'), E_{\ell}^{*}E_{\ell}]  $ is a pseudodifferential operator
of order $-1$ then $ [ a_{r,r}(t'), E_{\ell}^{*}E_{\ell}] D_{t_{r}}$ is a pseudodifferential operator of order zero.
The left hand side of the above estimate can bounded by
\begin{linenomath}
	\begin{multline*}
	 \sum_{\nu =1}^{q-1}
	\sum_{j_{2} =1}^{n_{1}} \sum_{p_{1} =1}^{m} 
	C_{4,\delta_{1}} C_{a}^{2(\nu +1)} q^{2\nu}
	\|P^{j_{2}}  D^{q-\nu-1}_{t_{r}}  D_{x_{p_{1}}} u\|_{0}^{2}
	\\
	+
	 \sum_{\nu =1}^{q-1}
	\sum_{j_{2} =1}^{n_{1}} \sum_{p_{1} =1}^{m} \sum_{p_{2} =1}^{m}
	C_{4,\delta_{1}} C_{a}^{2(\nu +1)} q^{2\nu}
	\| a_{n+p_{2},j_{2}} (t)  D^{q-\nu-1}_{t_{r}}  D_{x_{p_{1}}} D_{x_{p_{2}}}u\|_{0}^{2}
	\\
	+
	\sum_{\nu =1}^{q-1}
	 \sum_{p_{1} =1}^{m}
	m \| a_{n+p_{2},r} (t)  D^{q}_{t_{r}}  D_{x_{p_{1}}} u\|_{0}^{2}
	+
	2\| D^{q}_{t_{r}} u\|_{0}^{2} + \delta_{1} \| P^{r} D^{q}_{t_{r}} u\|_{0}^{2},
	\end{multline*}
\end{linenomath}
where $\delta_{1}$ is a suitable small constant as before and $C_{4,\delta_{1}}$ is a positive constant depending
of $\delta_{1}^{-1}$ but independent of $q$ and $\nu$. 
The last term on the right hand side can be absorbed on the left hand side of \eqref{eq:Est_Dt}.
The second to last term can be estimate as in \eqref{eq:Est_Rm} obtaining a
term that gives analytic growth and a term which can be absorbed
on the left hand side of \eqref{eq:Est_Dt}.
Concerning the other terms, they were already dealt several times,
more precisely the second and the third term can be estimate as done in \eqref{eq:Est_H12_1},
producing terms that give analytic growth directly or terms,
as the first, that can be estimated restarting the process via the subelliptic estimate. 
About these terms, as showed previously, after a suitable number of iterations,
they will give terms of the form \eqref{eq:J_6-iter}, that gives analytic growth.\\
Summing up, the term $J_{2}$ can be estimate by terms that give analytic growth,
modulo terms that can be absorbed on the right hand side of  \eqref{eq:Est_Dt}.

\bigskip
This ends the analysis of the terms on the right hand side of \eqref{eq:Est_T_S1},
$J_{1}, \dots, J_{7}$. We have obtained that all of them gives analytic growth rate
modulo terms that can be absorbed on the left hand side of  \eqref{eq:Est_Dt}.
We can conclude that
\begin{linenomath}
	\begin{equation*}
	\|D^{q}_{t_{r}} u \|_{\varepsilon}^{2}+ \sum_{j=1}^{N} \| P^{j} D^{q}_{t_{r}} u\|^{2}_{0} +  \sum_{j=1}^{n}\|P_{j}D^{q}_{t_{r}} u\|_{-1}^{2}
	\leq C^{2(q+1)} q!^{2}, 
	\end{equation*} 
\end{linenomath}
for all $r\in \{n_{1}+1,\dots, n   \}$,
where $C$ is a suitable positive constant independent of $q$.
This implies that
\begin{linenomath}
	\begin{equation*}
	\|\Delta_{t''}^{q}u \|_{0} \leq C_{1}^{q+1} q!^{2}, \qquad \forall q\in \mathbb{Z}_{+},
	\end{equation*}
\end{linenomath}
where $C_{1}$ is suitable large constant.\\
By this estimate and  the Proposition \ref{Pr_BCM},  we obtain that 
the points of the form $(t', t'',x, 0,\tau'', 0)$ don't belong to $WF_{a}(u)$.
In view of just shown, the estimate \eqref{eq:Dx_Est}, and the Remark \ref{Rmk2}, we 
conclude that $WF_{a}(u)$ is empty, i.e. $u \in C^{\omega}(\mathbb{T}^{N})$, 
which proves our statement, i.e. that the operator $P$, \eqref{Op_Th2}, is globally analytic hypoelliptic.
%
%
\section{Appendix(Some examples)}
\renewcommand{\theequation}{\thesection.\arabic{equation}}
\setcounter{equation}{0} \setcounter{theorem}{0}
\setcounter{proposition}{0} \setcounter{lemma}{0}
\setcounter{corollary}{0} \setcounter{definition}{0}
\setcounter{remark}{1}
%
In this section we collect some examples of differential operators which belong/ not belong to the class 
described by the assumptions $\mathbf{(A1)}$, $\mathbf{(A2)}$ and $\mathbf{(A3)}$. 
\vskip-5mm
\textbf{Case of sums of squares of H\"ormander type}.~\par
\vskip-6mm
The assumptions allow to include in the class under investigation the global analog of
some well known examples:
\begin{example}
Let $n_{1}=1$, $n=2$ and $m=1$. We consider the sum of squares operator
\begin{linenomath}
	\begin{equation*}
	D_{t_{1}}^{2} + a^{2}(t_{1})D_{t_{2}}^{2} + b^{2}(t_{1}) D_{x}^{2},
	\end{equation*}
\end{linenomath}
where $a$ and $b$ are real analytic functions not identically zero.\\
If we assume that $|b(t_{1})| \leq Ca^{2}(t_{1})$, where $C$ is a positive constant,
then the operator not only satisfies the assumptions $\mathbf{(A1)}$ and $\mathbf{(A2)}$,
but also \eqref{A3_1}, so it belongs to the class studied.\\
It can be seen as the globally defined version of
\vspace{-1.8em}
\begin{itemize}
	\item[i)] an elliptic operator if neither $a$ nor $b$ vanish;
	\item[ii)] the Baouendi-Goulaouic operator if  $a(t_{1})$ does not vanish and $b(t_{1})$ vanishes;
	              we recall that the local version of the Baouendi-Goulaouic operator (\cite{BG-72})
	              is given by $D_{t_{1}}^{2} + D_{t_{2}}^{2} + t_{1}^{2q} D_{x}^{2} $, $q\in \mathbb{Z}_{+}$;
	              the global analytic hypoellipticity of this operator was already studied in \cite{ch94}, \cite{Tart_96} and \cite{BC-2022};
	\item[iii)] a special version of the Ole\u \i nik-Radkevi\v c operator if $b(t_{1})$ vanishes only where $a(t_{1})$ vanishes;
	                we recall that the local version of the Ole\u \i nik-Radkevi\v c operator (\cite{OR1971}, \cite{OR1973}) is given
	                by $D_{t_{1}}^{2} +t^{2(p-1)}_{1} D_{t_{2}}^{2} + t_{1}^{2(q-1)} D_{x}^{2} $, $p,q\in \mathbb{Z}_{+} $, $1< p\leq q$
	                (the sharp local Gevrey regularity was studied in \cite{BT-97} and \cite{Ch_14}.) We point out that 
	                if in a neighborhood of the origin $a(t_{1})= \mathscr{O}(t_{1}^{p-1})$ and  $b(t_{1})= \mathscr{O}(t_{1}^{q-1})$
	                due to the assumption \eqref{A3_1} then the Ole\u \i nik-Radkevi\v c operator belongs to
	                the class studied if $q \geq 2p -1$; so our assumptions do not allow to cover all cases.
\end{itemize}
\end{example}
\begin{example}
	Let $n_{1}=1$, $n=2$ and $m=2$. We consider the sum of squares operator
	\begin{linenomath}
		\begin{equation}\label{Op_Ex2}
		D_{t_{1}}^{2} + a^{2}(t_{1})D_{t_{2}}^{2} + b^{2}(t_{1}) \left( D_{x_{1}}^{2}+ D_{x_{2}}^{2}\right) + c^{2}(t_{1},t_{2})  D_{x_{1}}^{2} + d^{2}(t_{1},t_{2})  D_{x_{2}}^{2},
		\end{equation}
	\end{linenomath}
where $a$, $b$, $c$ and $d$ are real analytic functions not identically zero.\\
We assume that $|b(t_{1})| \leq C a^{2}(t_{1})$, $|c(t_{1},t_{2})| \leq C a^{2}(t_{1})$ and $|d(t_{1},t_{2})| \leq Ca^{2}(t_{1}) $,
where $C$ is a positive constant, thus the operator satisfies the assumptions $\mathbf{(A1)}$, $\mathbf{(A2)}$
and \eqref{A3_1}, so it belongs to the class studied.\\
It can be seen as the globally defined version of the following operators:
\vspace{-1.8em}
	\begin{itemize}
		\item[i)]Assume that $a(t_{1})$ does not vanish, the assumption \eqref{A3_1}
		            implies that it is a generalization of the following operator
	            	  \begin{linenomath}
		         	    \begin{equation*}
			                D_{t_{1}}^{2} + D_{t_{2}}^{2} + b^{2}(t_{1}) \left( D_{x_{1}}^{2}+ D_{x_{2}}^{2}\right) + c^{2}(t_{1})  D_{x_{1}}^{2} + d^{2}(t_{1})  D_{x_{2}}^{2},
			            \end{equation*}
		               \end{linenomath}
		               here $a(t_{1})\equiv 1$. If we assume that $b$, $c$ and $d$ vanish at the origin, more precisely
		               $b(t_{1})= \mathscr{O}(t_{1}^{r-1})$, $c(t_{1})= \mathscr{O}(t_{1}^{p-1})$ and $d(t_{1})= \mathscr{O}(t_{1}^{q-1})$,
		               $r,p, q\in \mathbb{Z}_{+}$, $1<r<p<q$, then the operator is the glolly defined version of the operator
		               \begin{linenomath}
		               	$$
		               	 D_{t_{1}}^{2} + D_{t_{2}}^{2} + t_{1}^{2(r-1)}\left( D_{x_{1}}^{2}+ D_{x_{2}}^{2}\right) + t_{1}^{2(p-1)}  D_{x_{1}}^{2} + t_{1}^{2(q-1)}  D_{x_{2}}^{2};
		               	 $$
		               \end{linenomath}
		               which has been showed in \cite{ABM-16} to violate the local Treves conjecture (i.e. is not locally analytic hypoelliptic.)    
		               In this case the operator belong to the class studied in \cite{ch94}, \cite{BC-2022}, see also \cite{C-22}.
		\item[ii)] Assume that $a(t_{1})$ vanish at the origin, more precisely that $a(t_{1}) =t^{\ell}\widetilde{a}(t_{1})$, where $\widetilde{a}(0)\neq 0$
		               and $\ell \in \mathbb{Z}_{+}$. The assumption \eqref{A3_1} implies that $b,c,d= \mathscr{O}(t^{2\ell})$.
		               Assume that $ b = \mathscr{O}(t_{1}^{2\ell+r-1}) $,  $ r \in \mathbb{Z}_{+} $, $ r > 1 $, 
		               $ t_{1}^{-2\ell} c(t_{1},t_{2}) = \mathscr{O}(t_{2}^{p-1}) $, and that $ t_{1}^{-2\ell} d(t_{1},t_{2}) = \mathscr{O}(t_{2}^{q-1}) $, 
		               $ p $, $ q \in \mathbb{Z}_{+} $, $ 1 < r < p < q$.\\
		               Then the operator \eqref{Op_Ex2} is the global analog of the operator
		               \begin{linenomath}
		               	\begin{equation*}
		               	\hspace{2.6em}D_{t_{1}}^{2} + t_{1}^{2\ell} D_{t_{2}}^{2} + t_{1}^{2(2\ell+r-1)}\left(D_{x_{1}}^{2} +
		               	D_{x_{2}}^{2}\right) + t_{1}^{4\ell}\Big( t_{2}^{2(p-1)} D_{x_{1}}^{2} + t_{2}^{2(q-1)}
		               	D_{x_{2}}^{2} \Big) .
		               	\end{equation*}
		               \end{linenomath}
		               Although at the moment, as far as we know, there is no explicit proof, 
		               it is believed that this operator is not locally analytic hypoelliptic,
		               this should be a consequence of the presence of a non-symplectic Treves stratum.
		               We point out that in this case the operator \eqref{Op_Ex2} does not belong to the class
		               studied by Cordaro and Himonas \cite{ch94} but belong to the class studied in \cite{BC-2022}.
	\end{itemize}
\end{example}
\begin{remark}
	\label{Rk:Appendix}
A class of globally analytic hypoelliptic sums of squares that don't satisfy
the assumption $\mathbf{(A3)}$. 
More precisely, the assumption \eqref{A3_1} excludes the following operators
	\begin{linenomath}
		\begin{equation}\label{OP_Rmk}
		D_{t_{1}}^{2} + a^{2}(t_{1})D_{t_{2}}^{2} + b^{2}(t_{1}) \left( D_{x_{1}}^{2}+ D_{x_{2}}^{2}\right) + c^{2}(t_{1},t_{2})  D_{x_{1}}^{2} + d^{2}(t_{1},t_{2})  D_{x_{2}}^{2},
		\end{equation}
	\end{linenomath}
where $ a $ vanishes at the origin,
$ a(t_{1}) = t_{1}^{\ell} \tilde{a}(t_{1}) $, $ \ell \in \mathbb{Z}_{+} $ and $ \tilde{a}(0)
\neq 0$,  $ b = \mathscr{O}(t_{1}^{\ell+r-1}) $,  $ r \in \mathbb{Z}_{+} $ and $ r > 1 $,
$ t_{1}^{-\ell} c(t_{1},t_{2}) = \mathscr{O}(t_{2}^{p-1}) $ and  
$ t_{1}^{-\ell} d(t_{1},t_{2}) = \mathscr{O}(t_{2}^{q-1}) $,  $ p $, $ q \in \mathbb{Z}_{+} $
and $ 1 < r < p < q$. 
This operator is the global analog of 
\begin{linenomath}
$$ 
D_{t_{1}}^{2} + t_{1}^{2\ell} D_{t_{2}}^{2} + t_{1}^{2(\ell+r-1)}\left(D_{x_{1}}^{2} +
D_{x_{2}}^{2}\right) + t_{1}^{2\ell}\Big( t_{2}^{2(p-1)} D_{x_{1}}^{2} + t_{2}^{2(q-1)}
D_{x_{2}}^{2} \Big).
$$
\end{linenomath}
It has a symplectic characteristic real analytic manifold having a
symplectic stratification according to Treves. We recall that has been proved to
violate Treves conjecture in \cite{BM-16}, so it is not locally analytic hypoelliptic.
This operator does not belong to the class studied in \cite{ch94}.
In \cite{BC-2022} was showed that this operator is globally analytic hypoelliptic,
see also \cite{C-22}. 
To show the global analytic hypoellipticity  of \eqref{OP_Rmk}, it is enough to assume
that $|b(t_{1}) | \leq C |a(t_{1})|$, $|c(t_{1},t_{2}) |\leq  C |a(t_{1})|$ and $|d(t_{1},t_{2})| \leq  |a(t_{1})|$,
where $C$ is a positive constant, i.e. the assumption $(1.7)$ in \cite{BC-2022}.
So, in the case of sums of squares the assumption \eqref{A3_1} 
(i.e. that $|b(t_{1}) | \leq C |a^{2}(t_{1})|$, $c(t_{1},t_{2}) \leq  C|a^{2}(t_{1})|$ and $d(t_{1},t_{2}) \leq C|a^{2}(t_{1})|$ )
is stronger with respect to the assumption $(1.7)$ in \cite{BC-2022}. 
At the present, however, to treat the more general case of  H\"ormander-Ole\u{\i}nik-Radkevi\v c operators,
we are not able to make our hypothesis less strong. We think that in this situation the special form
in which these operators can be written, i.e. as sums of squares,
plays a role in the hypotheses to be assumed.
We also point out that in the situation of sums of squares, the 
Ole\u{\i}nik-Radkevi\v c-H\"ormander subelliptic estimate, see \eqref{eq:B-Est},
is weaker than the sharp Rothschild-Stein estimate, \cite{RS},
which is used in \cite{BC-2022}.

\end{remark}
\vskip-5mm
\textbf{Case of operators that are not of H\"ormander type}.~\par
\vskip-6mm
We give some examples of operators of the form \eqref{Op_ORH_1} 
which can not be written as sum of squares operator,
i.e. it doesn't belong to the H\"ormander's class of operators.
In order to it we follow the ideas in \cite{OR1971}.\\
Let us first recall some known facts (for more details on the subject see \cite{Marshall_Book}.)
We denote by $P_{d,n}$ the set  of homogeneous polynomials of degree $d$ in $n$ variables with coefficients in $\mathbb{R}$
which are positive semidefinite (i.e. non-negative) and by $\Sigma_{d,n}$ the subset of $P_{d,n}$ consisting of
all homogeneous polynomials that are sum of squares.
The following result is due to Hilbert
\begin{theorem}[\cite{Hilbert}]
For $d$ even, $P_{d,n}= \Sigma_{d,n}$ if and only if $n\leq 2$ or $d=2$ or $n=3$ and $d=4$.
\end{theorem}
\vspace{-1.5em}
The first example of a polynomial in two variables that can not be written as sum of squares is due to Motzkin \cite{Motzkin}:
\begin{linenomath}
	\begin{equation*}
	M_{0}(y,z)= 1 + y^{4}z^{2} + y^{2} z^{4} - 3y^{2}z^{2}.
	\end{equation*}
\end{linenomath}
Using $M_{0}$ can be constructed a polynomial in $P_{6,3}\setminus \Sigma_{6,3}$:
\begin{linenomath}
	\begin{equation}\label{eq:Motz}
	M(x,y,z)=x^{6}M_{0}\left(\frac{y}{x},\frac{z}{x}\right)= x^{6} + y^{4}z^{2} + y^{2} z^{4} - 3x^{2}y^{2}z^{2}.
	\end{equation}
\end{linenomath}
and more in general $M_{d}(x,y,z)=x^{d}M_{0}\left(\frac{y}{x},\frac{z}{x}\right)$,
where $d$ is even and greater than six,  $M_{d} \in P_{d,3}\setminus \Sigma_{d,3}$, named  homogenized Motzkin polynomials.\\
We recall the following result 
\begin{lemma}[Lemma 2.6.5 in \cite{OR1971}]\label{L_Hor}
	Let
	\begin{linenomath}
		\begin{equation*}
		A(x,y,z)= M(x,y,z) +P_{1}(x,y,z),
		\end{equation*}
	\end{linenomath}
	where $P_{1}(x,y,z)$ is an infinitely differentiable function which vanishes at the origin together with all its derivatives up to the sixth order inclusively.\\
	The infinitely differentiable function $A(x,y,z)$ is not representable in the form of finite sum of squares of infinitely differentiable
	functions in any neighborhood of the origin.
\end{lemma}
\begin{example}\label{Ex1}
Let $n_{1}=n=3$ and $m=1$. We consider the operator
\begin{linenomath}
	\begin{equation}\label{Ex_Op_2_1}
	P(t,D_{t},D_{x})=\Delta_{t}+ \widetilde{M}(t)D^{2}_{x} +d(t,x),  
	\end{equation}
\end{linenomath}
where $d\in C^{\omega}(\mathbb{T}^{4})$, $t=(t_{1},t_{2},t_{3})$, and
\begin{linenomath}
	\begin{equation}\label{Motz_Ex}
	\widetilde{M}(t_{1},t_{2},t_{3})= M\left(\sin t_{1}, \sin t_{2}, \sin t_{3}\right),
	\end{equation}
\end{linenomath}
where $M$ is the homogenized Motzkin polynomial, \eqref{eq:Motz}, where we replace $x$, $y$ and $z$
by $\sin t_{1}$, $\sin t_{2}$ and $\sin t_{3}$ respectively. We remark that $\widetilde{M}(t)\geq 0$ and 
$\widetilde{M}(2k\pi)=0$, for every $k \in \mathbb{Z}$.\\
We show that $P$ does not belong to H\"ormaner's class.
If $P$ belonged to it, we should have that
\begin{linenomath}
\begin{equation}\label{Eq:Ex1}
    	\Delta_{t}+ \widetilde{M}(t) D_{x}^{2}
    	=\sum_{j=1}^{r} X_{j}^{2} +i X_{0},
\end{equation}
\end{linenomath}
where $X_{j}$, $j=0,1,\dots, r$, are suitable vector fields of the form
\begin{linenomath}
    	\begin{equation*}
    	X_{j}= \sum_{\ell =1}^{3} a_{j,\ell}(t) D_{t_{\ell}} +  a_{j,4}(t) D_{x}. 
    	\end{equation*}
 \end{linenomath}
Equating the coefficients of second derivatives on the left and right hand side of \eqref{Eq:Ex1},
it follow that
\begin{linenomath}
 \begin{equation*}
    	\sum_{j=1}^{r} \left( a_{j,4} (t)\right)^{2}= \widetilde{M}(t) , 
    	\quad \forall t.
\end{equation*}
\end{linenomath}
Now, if we represent $\widetilde{M}(t) $ in the form of partial Taylor series expansion
in a neighborhood of $t=0$, we should have that
\begin{linenomath}
   \begin{equation*}
     \sum_{j=1}^{r} \left( a_{j,4} (t)\right)^{2}= M(t) +P_{1}(t),
   \end{equation*}
 \end{linenomath}
where $M(t)$ is the homogenized Motzkin polynomial, \eqref{eq:Motz}, and $P_{1}(t)$
is an infinitely differentiable function which vanishes at the origin together with all its derivatives up to the sixth order inclusively.
But this impossible by the Lemma \ref{L_Hor}. So, the operator $P$ can not be written as sum of squares.
 
Moreover, we observe that the operator $P$, satisfies the assumptions $\mathbf{(A1)}$ and $\mathbf{(A2)}$, case
$n_{1}=n$, here $P^{k}= D_{t_{k}}$, $P^{4}= \widetilde{M}(t) D_{x}$ and $P_{k} = \frac{\partial \widetilde{M} }{\partial t_{k}} (t) D_{x}^{2} $, $k=1,2,3$.
By the Theorem \ref{Th_1} it is globally analytic hypoelliptic.
\end{example}
\begin{example}Let $n_{1}=3$, $n=4$ and $m=1$. We consider the operator
\begin{linenomath}
 \begin{equation}\label{Ex_Op_2_2}
	P(t,D_{t},D_{x})=\Delta_{t'}+ \widetilde{M}(t')D_{t_{4}}^{2} + f(t) D^{2}_{x},  
\end{equation}
\end{linenomath}
where $t=(t_{1},t_{2},t_{3},t_{4})$, $t'=(t_{1},t_{2},t_{3})$, $\widetilde{M}(t') $ is as in \eqref{Motz_Ex}, and $f(t)$ is a non-negative real analytic function
not identically zero. 
If we assume that $f(t) = \widetilde{f}(t)\widetilde{M}^{2}(t')$, where $\widetilde{f} \in C^{\omega}(\mathbb{T}^{4}) $ and $\widetilde{f}(t)\geq 0$,
then the assumption \eqref{A3_1} is satisfied. The operator \eqref{Ex_Op_2_2} satisfies the assumptions $\mathbf{(A1)}$, $\mathbf{(A2)}$ and $\mathbf{(A3)}$,
by the Theorem \ref{Th_2} it is globally analytic hypoellptic.
We point out that following the same argument used in the Example 3,
can be showed that the operator \eqref{Ex_Op_2_2}
does not belong to the H\"ormander's class.
\end{example}
%
%

%

%
\end{document}